\documentclass[onefignum,onetabnum]{siamart190516}

\usepackage{geometry}
\allowdisplaybreaks[4]
\usepackage{lipsum}
\usepackage{amsfonts}
\usepackage{graphicx}
\usepackage{epstopdf}
\usepackage{algorithmic}
\ifpdf
  \DeclareGraphicsExtensions{.eps,.pdf,.png,.jpg}
\else
  \DeclareGraphicsExtensions{.eps}
\fi

\newsiamremark{remark}{Remark}
\newsiamremark{hypothesis}{Hypothesis}
\crefname{hypothesis}{Hypothesis}{Hypotheses}
\newsiamthm{claim}{Claim}

\headers{An Example Article}{D. Doe, P. T. Frank, and J. E. Smith}

\title{An Example Article\thanks{Submitted to the editors DATE.
\funding{This work was funded by the Fog Research Institute under contract no.~FRI-454.}}}

\author{Dianne Doe\thanks{Imagination Corp., Chicago, IL 
  (\email{ddoe@imag.com}, \url{http://www.imag.com/\string~ddoe/}).}
\and Paul T. Frank\thanks{Department of Applied Mathematics, Fictional University, Boise, ID 
  (\email{ptfrank@fictional.edu}, \email{jesmith@fictional.edu}).}
\and Jane E. Smith\footnotemark[3]}

\usepackage{amsopn}

\usepackage{subcaption}

\newcommand{\bfd}{{\bf d}}
\newcommand{\bfn}{{\bf n}}
\newcommand{\bfzeta}{\boldsymbol{\zeta}}

\newcommand{\vertiii}[1]{{\left\vert\kern-0.25ex\left\vert\kern-0.25ex\left\vert #1
    \right\vert\kern-0.25ex\right\vert\kern-0.25ex\right\vert}}
\newcommand{\verti}[1]{{\left\vert #1
    \right\vert}}
\ifpdf
\headers{Trilinear IFEM for 3D Interface Problems}{R. Guo and X. Zhang}

\title{Solving Three-Dimensional Interface Problems with Immersed Finite Elements: A-Priori Error Analysis}

\author{Ruchi Guo\thanks{Department of Mathematics, The Ohio State University, Columbus, OH 43210
  \email{guo.1778@osu.edu}.}
\and Xu Zhang\thanks{Department of Mathematics, Oklahoma State University, Stillwater, OK 74078
  \email{xzhang@okstate.edu}.}}

\newcommand{\aver}[1]{\left\{\!\!\left\{#1\right\}\!\!\right\}}

\newcommand{\jump}[1]{\left[\!\left[#1\right]\!\right]}

\DeclareFontEncoding{FMS}{}{}
\DeclareFontSubstitution{FMS}{futm}{m}{n}
\DeclareFontEncoding{FMX}{}{}
\DeclareFontSubstitution{FMX}{futm}{m}{n}
\DeclareSymbolFont{fouriersymbols}{FMS}{futm}{m}{n}
\DeclareSymbolFont{fourierlargesymbols}{FMX}{futm}{m}{n}
\DeclareMathDelimiter{\tbar}{\mathord}{fouriersymbols}{152}{fourierlargesymbols}{147}

\begin{document}

\maketitle

\begin{abstract}
  In this paper, we develop and analyze a trilinear immersed finite element method for solving three-dimensional elliptic interface problems. The proposed method can be utilized on interface-unfitted meshes such as Cartesian grids consisting of cuboids. We establish the trace and inverse inequalities for trilinear IFE functions for interface elements with arbitrary interface-cutting configuration. Optimal a priori error estimates are rigorously proved in both energy and $L^2$ norms. Numerical examples are provided not only to verify our theoretical results but also to demonstrate the applicability of this IFE method in tackling some real-world 3D interface models. 
\end{abstract}



\section{Introduction}
Interface problems are ubiquitous. Many real-world applications in fracture mechanics, fluid mechanics, and material science involve multiple mediums and can be considered as three-dimensional (3D) interface problems. For mathematicians and computational scientists, partial differential equations (PDEs) are often used to model these problems. Usually, these governing equations have discontinuous coefficients that represent the different material properties. 

To solve interface problems, in general, there are two classes of numerical methods. The first class of methods uses interface-fitted meshes, i.e., the mesh must be tailored to fit the interface. Methods of this type include classical finite element method (FEM) \cite{1998ChenZou}, discontinuous Galerkin method \cite{2002ArnoldBrezziCockburnMarini}, and virtual element method \cite{2017ChenWeiWen}, to name only a few. The second class of numerical methods uses interface-unfitted meshes that are independent of the interface.  Structured meshes such as the Cartesian mesh are often used in these methods. An immediate benefit of these unfitted-mesh methods is the avoidance of re-mesh when solving a dynamic problem with evolving interfaces. For example, this feature can be particularly advantageous in simulating multi-phase fluid flow \cite{1997HouLiOsherZhao}, crystal growth \cite{2004BanschHauberLakkisLiVoigt}, solving geometric inverse problems \cite{GuoLinLin2017} and so on. We refer readers to \cite{2006LiIto} for various applications. Moreover, the mesh generation can be especially challenging in the 3D case since the geometry and topology can be rather complicated such as those in biomedical image \cite{2009FormaggiaQuarteroniVeneziani} and geophysical image \cite{2014DassiPerottoFormaggiaRuffo}. 

In the past few decades, there are many numerical methods introduced for solving interface problems based on unfitted meshes. In the finite difference framework, there are Peskin's immersed boundary method \cite{2002Peskin}, immersed interface method \cite{1994LevequeLi}, matched interface and boundary method \cite{2006ZhouZhaoFeigWei}, to name only a few. In the framework of finite element methods (FEMs), there are general FEM \cite{1994BabuskaCalozOsborn}, Cut-FEM \cite{2015BurmanClaus}, multi-scale FEMs \cite{2010ChuGrahamHou,1999HouWuCai}, extended FEM \cite{2001DolbowMoesBelytschko},  partition of unity method  \cite{1996BabuskaMelenk}, and immersed FEM \cite{1998Li}, etc.

The immersed FEM was first developed in \cite{1998Li} for solving one-dimensional (1D) elliptic interface problem, in which the lowest order IFE function was developed and analyzed. The fundamental idea is to construct some special shape functions capturing the jump behavior of the exact solution. Since then, IFE method has been extended to higher-order approximation \cite{2009AdjeridLin, 2017CaoZhangZhang} in 1D, and two-dimensional (2D) interface problems \cite{2016GuoLin,2019GuoLin,2016GuzmanSanchezSarkis,2019HeZhang,2015LinLinZhang, 2019LinSheenZhang}, and 3D interface problems \cite{2005KafafyLinLinWang, 2010VallaghePapadopoulo, 2016HanWangHeLinWang, 2020GuoLin}. Besides the classical second-order elliptic equation, IFE methods have been applied in a wide variety of interface problems, such as the linear elasticity system \cite{2012LinZhang,2018GuoLinLin}, moving interface problems \cite{2013HeLinLinZhang,2013LinLinZhang1}, interface inverse problems \cite{GuoLinLin2017}, and stochastic interface models \cite{2016ZhangLiZhang}.

So far, most of IFE method in literature deal with 2D interface problems. Very few tackles the real 3D interface problems. In \cite{2005KafafyLinLinWang}, a linear IFE method was introduced on unfitted tetrahedral meshes, and was then used in \cite{2016HanWangHeLinWang} for simulating plasma–lunar surface interactions. In \cite{2010VallaghePapadopoulo}, a trilinear element was introduced on cuboidal mesh for solving the electroencephalography forward problem. However, there are no theoretical results for either of these methods. Recently, in \cite{2020GuoLin}, the authors reconstructed trilinear IFE functions on cuboidal meshes based on the actual interface surface. The unisolvency of the trilinear IFE functions was shown using the invertibility of a Sherman-Morrison matrix. The maximum angle condition was employed in the construction procedure to guarantee the optimal approximation capabilities of the trilinear IFE spaces, and the rigorous proof was also given through detailed geometrical analysis.

As most of the IFE spaces in the literature, the global IFE functions in \cite{2020GuoLin} are discontinuous across interface faces which can cause certain nonconformity and loss of convergence order if the standard Galerkin scheme is used. The partially penalized IFE (PPIFE) scheme has been widely used to address this issue \cite{2015LinLinZhang,2018GuoLinLin} in 2D situation, and the basic idea is to use interior penalties to handle discontinuities only on interface edges/faces. The PPIFE method was first introduced in \cite{2015LinLinZhang} for the 2D elliptic interface problem in which the analysis relies on piecewise $H^3$ regularity of the solution. Recently, through interpolation error analysis on the patch of interface elements, it was proved in \cite{2019GuoLinZhuang} that the errors decay optimally in both energy norm and $L^2$ norm requiring only the piecewise $H^2$ regularity of the solution. But to the best of our knowledge, due to the more challenging geometry, there is no theoretical analysis for the {\it a priori} error analysis for 3D interface problems.

This paper has two major contributions. The first one is to conduct the rigorous error analysis for the PPIFE method for 3D interface problems. The global degrees of freedom for the proposed IFE method are isomorphic to the standard continuous piecewise trilinear finite element space defined on the same mesh which is independent of the interface location and advantageous for moving interface problems. But due to the complexity of the geometrical configurations of interface elements and the corresponding IFE functions, the analysis can be very challenging. For example, fundamental inequalities such as the trace inequality and inverse inequality must be re-established for three-dimensional IFE functions. Nevertheless, the standard theoretical tools can barely be used due to the low regularity of the solution. In our analysis, we show the discrete extension operator used to construct IFE functions is stable regardless of interface location, and this stability severs as the foundation of the trace and inverse inequalities, which is also the key for the proposed PPIFE method to be stable for an arbitrary interface. Another challenge is the inconsistency of the numerical scheme due to the discontinuity of the trilinear IFE function across the interface surface. Thanks to the optimal error bound of the interpolation operator \cite{2020GuoLin}, we are able to show that the inconsistency term will not affect the overall accuracy, namely, there is no need to add penalties on the interface surface. The second contribution is the extensive investigation of the applicability of the proposed IFE method. In particular, we demonstrate that it can be used to solve problems with various interface shape and topology. Moreover, we also investigate the implementation of the method for some real-world interface models where only the original cloud-point geometric data on the interface are available. In a realistic simulation, these raw data need to be used to generate a computational interface surface which can be further utilized by the proposed IFE method.

The rest of the paper is organized as follows. In Section 2, we present the three-dimensional interface problem and recall some geometrical properties of the 3D cuboidal meshes with interface surfaces. In Section 3, we present the trilinear IFE spaces and the partially penalized IFE method for solving 3D interface problems. In Section 4, we prove fundamental inequalities including the trace and inverse inequalities of the trilinear IFE functions. In Section 5, we derive the \textit{a priori} error estimates of PPIFE solutions in both energy norm and $L^2$ norm. In Section 6, we present extensive numerical experiments not only to verify our theoretical results but to demonstrate how this IFE method can be applied to tackle the real-world 3D interface problems. A brief conclusion will be drawn in Section 7.

\section{Interface Models and Preliminary Results}
\label{sec:model}

Let $\Omega\subseteq\mathbb{R}^3$ be an open bounded domain. Without loss of generality, we assume that
$\Omega$ is separated into two subdomains $\Omega^-$ and $\Omega^+$ by a closed $C^2$ manifold $\Gamma\subseteq\Omega$ known as the interface. These subdomains contain different materials identified by a piecewise constant function $\beta(\mathbf{x})$ which is discontinuous across the interface $\Gamma$, i.e.,
\begin{equation*}
\beta(\mathbf{x})=
\left\{\begin{array}{cc}
\beta^- & \text{in} \; \Omega^- ,\\
\beta^+ & \text{in} \; \Omega^+,
\end{array}\right.
\end{equation*}
where $\beta^\pm>0$ and $\mathbf{x} = (x,y,z)$. We consider the following interface problem of the elliptic type on $\Omega$:
\begin{subequations}\label{model}
\begin{align}
\label{inter_PDE}
 -\nabla\cdot(\beta\nabla u)=f, &~~~~  \text{in} ~ \Omega^-  \cup \Omega^+, \\
 \jump{u}_{\Gamma} = 0, &~~~~  \text{on} ~ \Gamma \label{jump_cond_1}, \\
 \jump{\beta \nabla u\cdot \mathbf{n}}_{\Gamma}  = 0, & ~~~~\text{on}~\Gamma \label{jump_cond_2}, \\
  u=g, &~~~~\text{on} ~ \partial\Omega,
\end{align}
\end{subequations}
where $\jump{v}_\Gamma := (v|_{\Omega^+})_\Gamma-(v|_{\Omega^-})_\Gamma$, and $\mathbf{n}$ is the unit normal vector to $\Gamma$. For simplicity, we denote $u^s=u|_{\Omega^s}$, $s=\pm$, in the rest of this article. Here we only consider the homogenous jump condition, and the nonhomogeneous case can be treated by some enriched functions through the framework recently developed by Babu\v{s}ka et al. in \cite{2020AdjeridBabukaGuoLin}.

In this section, we first introduce some abstract spaces used throughout this article and recall some geometrical properties of the unfitted mesh for three-dimensional interface problems. Given an open subset $\tilde \Omega \subseteq \Omega$, let $H^{k}(\tilde \Omega)$ be the standard Hilbert spaces on $\tilde \Omega$ with the norm $\|\cdot\|_{k,\tilde \Omega}$ and the semi-norm $|\cdot|_{k,\tilde \Omega}$. In the case $\tilde \Omega^s := \tilde \Omega \cap \Omega^s \not = \emptyset, s = \pm$, we define the splitting Hilbert spaces
\begin{equation}
\label{split_Hspa}
PH^k(\tilde{\Omega}) = \{ u\in H^k(\tilde{\Omega}^{\pm}) ~:~  \jump{u}_{\Gamma\cap\tilde{\Omega}}=0 ~ \text{and} ~  \jump{\beta\nabla u\cdot\mathbf{ n}}_{\Gamma\cap\tilde{\Omega}}=0 \},
\end{equation}
where the definition implicitly implies the involved traces on $\Gamma\cap\tilde{\Omega}$ are well defined, with the associated norms and semi-norms defined as follows
\[
\|\cdot\|^2_{H^k(\tilde \Omega)}:=\|\cdot\|^2_{H^k(\tilde \Omega^+)}+\|\cdot\|^2_{H^k(\tilde \Omega^-)},~~~
|\cdot|^2_{H^k(\tilde \Omega)}:=|\cdot|^2_{H^k(\tilde \Omega^+)}+|\cdot|^2_{H^k(\tilde \Omega^-)}.
\]

In the following, we assume that $\Omega\subset\mathbb{R}^3$ is a cuboid domain, and $\mathcal{T}_h$ is a Cartesian cuboidal mesh of $\Omega$ where $h$ denotes the maximum length of the all cuboids. Denote $\mathcal{F}_h$, $\mathcal{E}_h$ and $\mathcal{N}_h$ as the collections of faces, edges, and nodes, respectively. We call an element $T\in \mathcal{T}_h$ an interface element if not all of its vertices locate on the same side of the interface $\Gamma$; otherwise, we treat it as a non-interface element. Similarly, we can define the interface faces and interface edges by the relative location of its vertices with the interface. Note that non-interface elements/faces/edges may still intersect with the interface due to large curvature of the segment of the interface, see the illustration in Figure \ref{fig:interface_elem_refine}. However, this issue can always be resolved by refining the mesh. 
Let $\mathcal{T}^i_h$/$\mathcal{F}^i_h$/$\mathcal{E}^i_h$ and $\mathcal{T}^n_h$/$\mathcal{F}^n_h$/$\mathcal{E}^n_h$ be the collections of interface and non-interface elements/faces/edges, respectively. Let $\kappa$ be the maximal curvature (principle curvature) of the interface surface $\Gamma$. Moreover, for each interface element $T\in \mathcal{T}^i_h$, we define its patch $\omega_T$ as
\begin{equation}
\label{patch}
\omega_T = \{ T'\in \mathcal{T}_h~:~ \overline{T'}\cap \overline{T} \neq \emptyset \}.
\end{equation}

Many unfitted-mesh methods rely on the assumption that the mesh size is sufficiently small such that the interface curve/surface is resolved enough \cite{2016GuoLin,2002HansboHansbo}. In this section, we provide a delicate approach to quantify how well the interface is resolved by a fixed mesh. In particular, our approach is to measure the flatness of the interface within each interface element in terms of the maximal angle between the normal vectors of the interface surface and its planar approximation. These fundamental geometric results will be used throughout this paper. First of all, we recall the so-called $r$-tubular neighborhood of a smooth manifold from \cite{1959Federer} which is actually a very useful concept in computational geometry \cite{2002MorvanThibert}.
\begin{lemma}[$r$-tubular neighborhood]
\label{tubular}
Given a smooth compact surface $\Gamma$ in $\mathbb{R}^3$, for each point $X\in\Gamma$, let $N_X(r)$ be a segment with the length $2r$ centered at $X$ and perpendicular to $\Gamma$. Then, there exists a positive $r>0$ such that $N_X(r)\cap N_Y(r) = \emptyset$
for any $X, Y\in \Gamma, X \not = Y$. Then the $r$-tubular neighborhood of $\Gamma$ is defined as the set $U_{\Gamma}(r)=\cup_{X\in\Gamma} N_X(r)$.
\end{lemma}

Define $r_{\Gamma}$ to be the largest $r$ such that Lemma \ref{tubular} holds, namely it corresponds to the largest $r$-tabular neighborhood, and this positive number $r_{\Gamma}$ is referred as the reach of the surface $\Gamma$ \cite{2002MorvanThibert}. We further note that the reach $r_{\Gamma}$ is only determined by the surface itself. Throughout this paper, we assume that the mesh size $h$ is sufficiently small such that the following hypotheses hold \cite{2020GuoLin}:
\begin{itemize}
  \item[(\textbf{H1})] $h<r_{\Gamma}/(3\sqrt{3})$.   
  \item[(\textbf{H2})] $h \kappa \leq 0.0288.$
  \item[(\textbf{H3})] The interface $\Gamma$ cannot intersect an edge $e\in\mathcal{E}_h$ at more than one point.
  \item[(\textbf{H4})] The interface $\Gamma$ cannot intersect a face $f\in\mathcal{F}_h$ at more than two edges.
\end{itemize}
These hypotheses basically ensure that the interface surface is sufficiently resolved by the unfitted mesh such that it is flat enough inside each interface element. Similar assumptions have been used in many unfitted-mesh methods such as \cite{2015BurmanClaus,2015GuzmanSanchezSarkisP1,2016GuoLin,2002HansboHansbo,2005KafafyLinLinWang}. In this work we make the bounds in (\textbf{H1}) and (\textbf{H2}) explicit and computable which can guide mesh generation in real computation. We refer readers to \cite{2020GuoLin} for the details of calculation of those bounds.

Now we are ready to describe the classification of the interface elements. Based on hypotheses (\textbf{H3}) and (\textbf{H4}), we claim that the interface surface can only intersect an element at six points the most. In fact, suppose that an element has 7 intersection points. According to (\textbf{H3}), these 7 points must be on 7 different edges. Since every edge is shared by two adjacent faces in the element, there are a total of 14 interface edges counting each edge twice from its sharing faces. A cuboid has six faces in total, which means there is at least one face containing at least 3 interface edges. This is contradicted to (\textbf{H4}). 

According to \cite{2020GuoLin}, when the interface is resolved sufficiently by an unfitted mesh there are only five possible interface element configurations as shown in Figure \ref{fig:subfig_interf}. Taking into account of rotation, the five types of interface cuboids have the following representatives: 
\begin{description}
\item [Type I interface element:] three intersection points on three edges
\item [Type II interface element:] four intersection points on four parallel edges
\item [Type III interface element:] four intersection points on two pairs of adjacent edges
\item [Type IV interface element:] five intersection points on five edges
\item [Type V interface element:] six intersection points on six edges
\end{description}
See Figure \ref{fig:subfig_interf} (\subref{inter_elem_case1}-\subref{inter_elem_case5}) for illustrations of all types of interface elements. 
This classification strategy can also be used in computation to efficiently determine the geometry configuration of each interface element by counting the number of interface points and the number of vertices on each side. The classification of interface elements is important in constructing IFE functions and numerical quadrature which we shall discuss later on.

\begin{figure}[!]
\centering
     \includegraphics[width=1.8in]{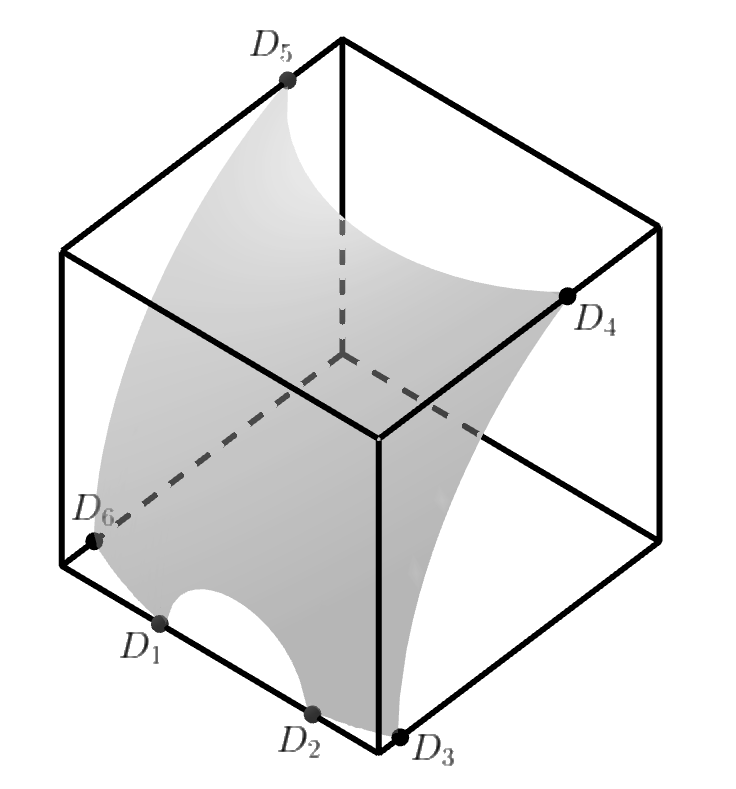}~~~
     \includegraphics[width=1.7in]{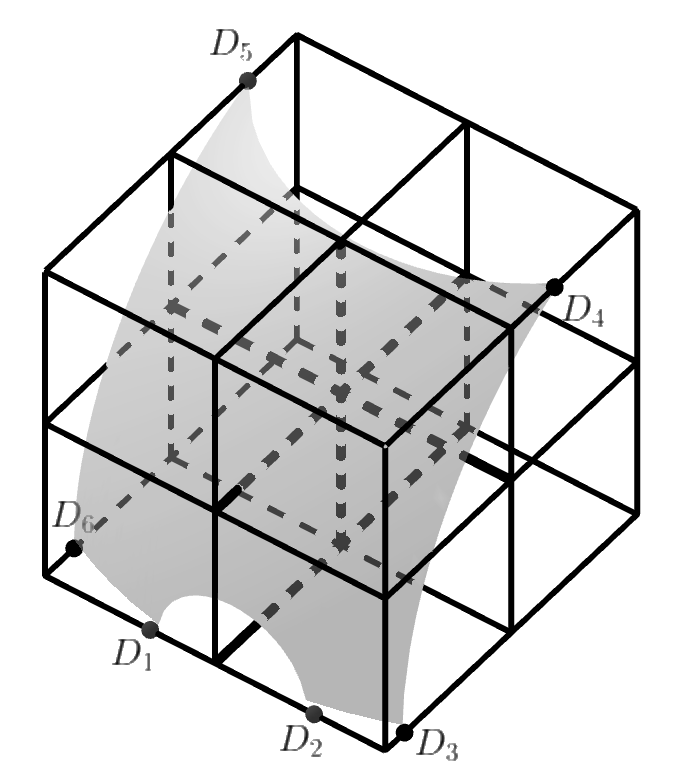}
\caption{Refine an interface element into eight congruent elements. Left: an interface element with an edge containing two intersection points (left). Right: a further partition such that this element satisfies the hypothesis (\textbf{H3}).} 
  \label{fig:interface_elem_refine} 
\end{figure}


\begin{figure}[h]
\centering
\begin{subfigure}{.3\textwidth}
     \includegraphics[width=1.5in]{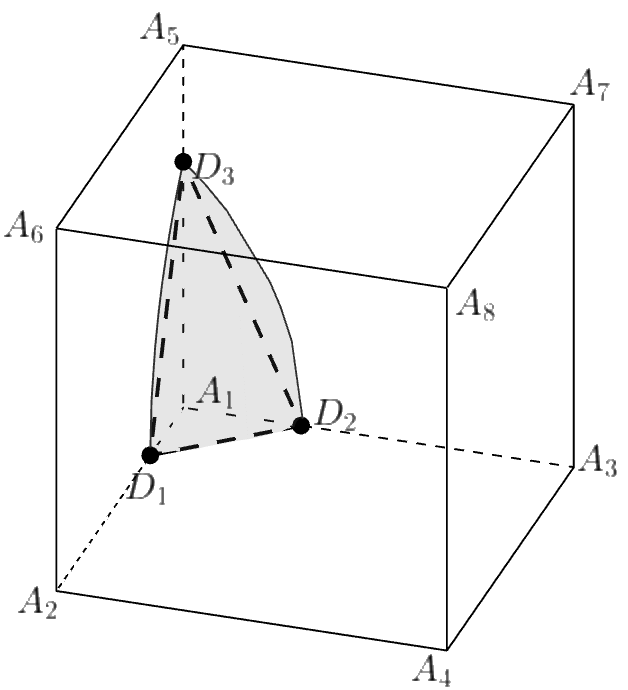}
     \caption{Type I: 3 intersection points}
     \label{inter_elem_case1} 
\end{subfigure}
~
\begin{subfigure}{.3\textwidth}
     \includegraphics[width=1.5in]{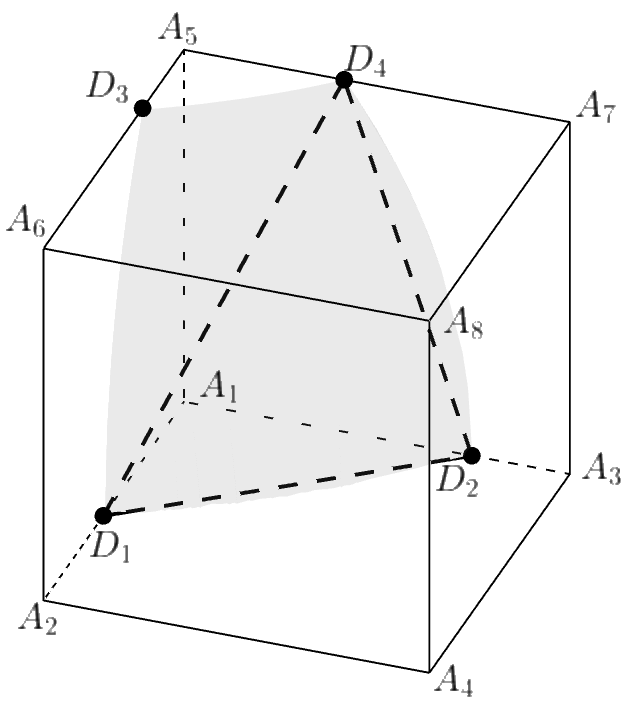}
     \caption{Type II: 4 intersection points on the adjacent edges}
     \label{inter_elem_case2} 
\end{subfigure}
~
 \begin{subfigure}{.3\textwidth}
     \includegraphics[width=1.5in]{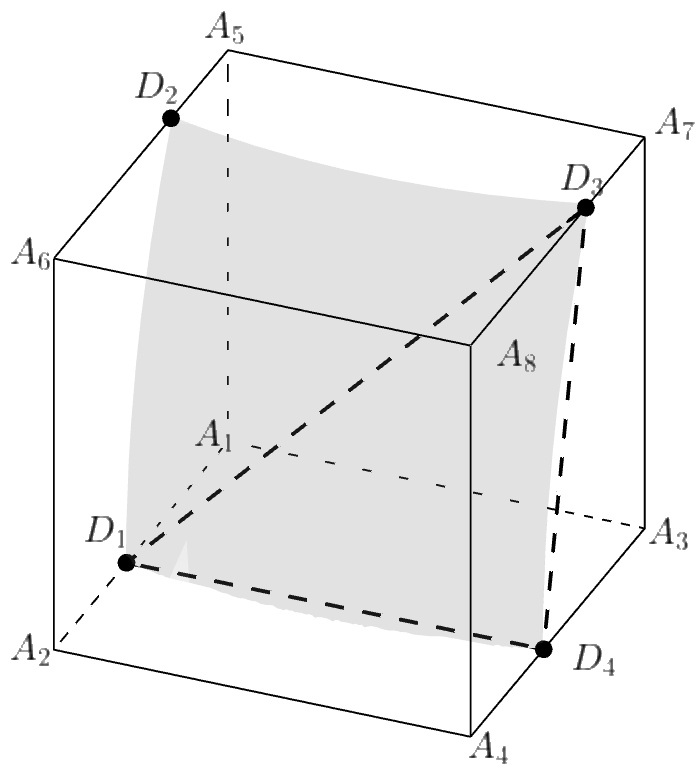}
     \caption{Type III: 4 intersection points on the opposite edges}
     \label{inter_elem_case3} 
\end{subfigure}
 \begin{subfigure}{.3\textwidth}
     \includegraphics[width=1.5in]{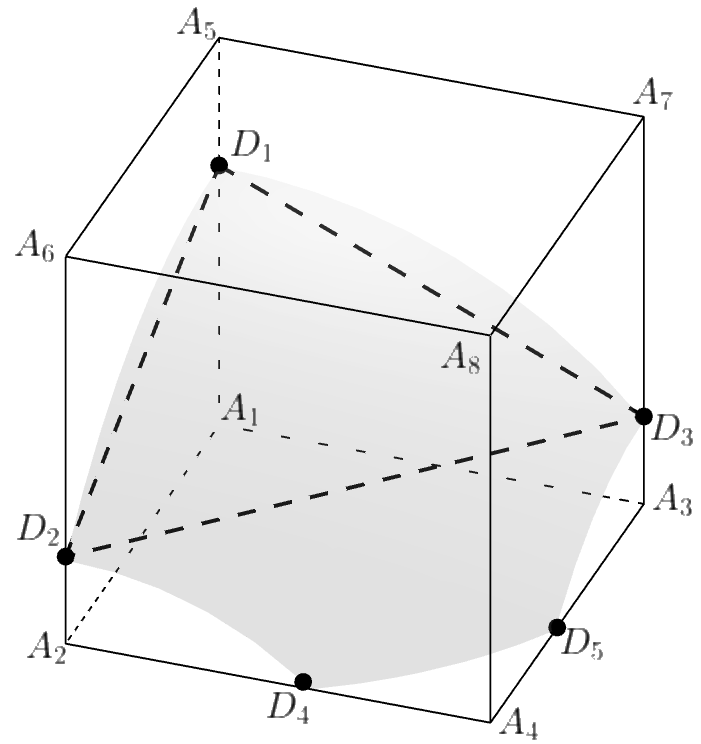}
     \caption{Type IV: 5 intersection points}
      \label{inter_elem_case4} 
\end{subfigure}
~~
\begin{subfigure}{.3\textwidth}
     \includegraphics[width=1.5in]{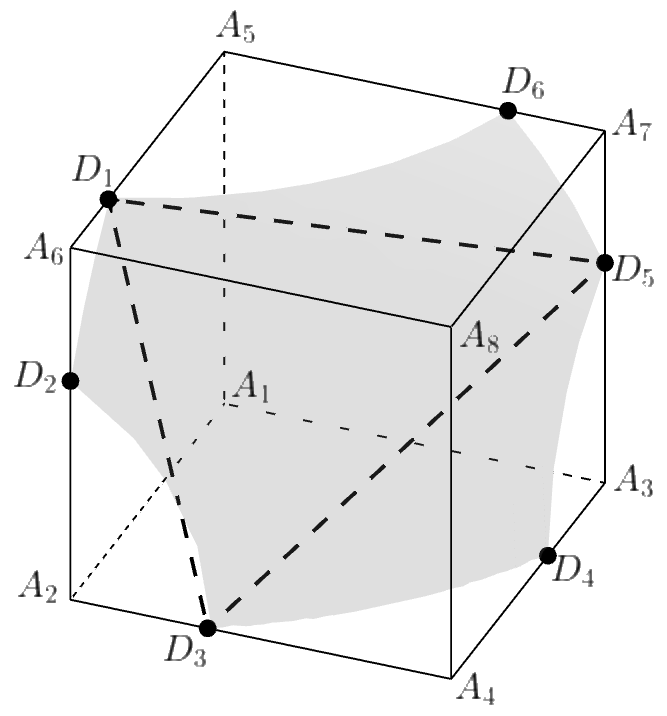}
     \caption{Type V: 6 intersection points}
      \label{inter_elem_case5} 
\end{subfigure}
     \caption{Possible Interface Element Configuration }
  \label{fig:subfig_interf} 
\end{figure}

Now we describe how to construct an approximate plane, denoted by $\tau_T$, for each interface element $T\in\mathcal{T}_h^i$ with sufficient geometric representation for the interface. This is done by constructing a triangle formed by three suitable intersection points such that its maximal angle is always bounded by $135^{\circ}$ regardless of the interface location \cite{2020GuoLin}. As shown by Figure \ref{fig:subfig_interf}, we follow the choice in \cite{2020GuoLin} to make this plane $\tau_T$ contain the following triangles: $\triangle D_1D_2D_3$ for Type I, $\triangle D_1D_2D_4$ for Type II, $\triangle D_1D_4D_3$ for Type III, $\triangle D_1D_2D_3$ for Type IV and $\triangle D_1D_3D_5$ for Type V. We emphasize that the choice may not be unique, and any triangle is acceptable as long as the maximal angle condition is satisfied. We refer readers to \cite{2017ChenWeiWen,2020GuoLin} for more details on the calculation of the maximal angles of these triangles. Note that the maximal angle condition is also widely used in standard finite element analysis which can be traced back to the early works of Babu\v{s}ka \cite{1974BabuskaAziz}. Under this choice of the plane $\tau_T$, we recall from \cite{2020GuoLin} the following optimal geometric error estimate.

\begin{lemma}
\label{lem_interf_element_est}
Let $\mathcal{T}_h$ be  a Cartesian mesh whose mesh size is small enough such that (\textbf{H1})-(\textbf{H4}) hold, then the following estimates hold for every point $X\in\Gamma\cap T$ (or every point $X\in\Gamma\cap \omega_T$):
\begin{subequations}
\label{interf_element_est_eq0}
\begin{align}
    & \| X - X_{\bot} \| \le 12.0927 \kappa h^2, \label{interf_element_est_eq0_dist} \\
    &  \mathbf{ n}(X) \cdot \bar{\mathbf{ n}} \ge 1 - 26.6121 \kappa^2 h^2,  \label{interf_element_est_prod}
\end{align}
\end{subequations}
where $X_{\bot}$ is the projection of $X$ onto $\tau_T$, $\mathbf{ n}(X)$ is the unit normal vector to $\Gamma$ at $X$, and $\bar{\mathbf{ n}}$ is the normal vector to $\tau_T$.
\end{lemma}

\begin{proof}
See the proof of Theorem \ref{lem_interf_element_est} of \cite{2020GuoLin}.
\end{proof}

A direct consequence of Lemma \ref{lem_interf_element_est} is the following lemma.
\begin{lemma}
\label{lem_interf_area}
Under the conditions of Theorem \ref{lem_interf_element_est}, there exists a constant $C$ independent of interface location and mesh size $h$ such that $meas({\Gamma\cap T}) \le Ch^2$.
\end{lemma}
\begin{proof}
Clearly, we have $meas(\tau_T)\le Ch^2$. Then using \eqref{interf_element_est_prod} we have
\begin{equation*}
meas({\Gamma\cap T}) = \verti{\int_{\Gamma\cap T} dS} = \verti{\int_{\tau_T} \frac{1}{\mathbf{ n}(X)\cdot\bar{\mathbf{ n}}} dS} \le Ch^2.
\end{equation*}
Finally for simplicity's sake, we shall employ a generic constant $C$ in the rest of this article which is independent of interface location, mesh size and discontinuous coefficients $\beta^{\pm}$ without explicitly mentioning in presentation. In addition the notation $\simeq$ denotes equivalence where the hidden constant $C$ has the same property.

\end{proof}


\section{Trilinear IFE Spaces and the IFE Method}
In this section we describe the trilinear IFE functions and the PPIFE method. In general IFE functions constructed by piecewise polynomials cannot satisfy the jump conditions exactly for an arbitrary interface surface. Different approximations of jump conditions have been proposed in the 2D case, see \cite{2016GuoLin,2019GuoLin,2008HeLinLin}. Most of the methods rely on the linear approximation of the interface curve constructed by simply connecting the intersection points, and then the approximate jump conditions are posed on this line. However this approach becomes obscure in 3D case since the intersection points, the number varying from three to six, may not be coplanar. Some early works of IFE functions use the approximation plane passing through the three points which has the shortest distance to the others for which we refer readers to \cite{2016HanWangHeLinWang,2005KafafyLinLinWang} for details. Besides, a level-set approximate approach was used in \cite{2010VallaghePapadopoulo}. But to our best knowledge, these works on 3D IFE functions are in lack of theoretical foundation. Recently, the authors in \cite{2020GuoLin} proposed a new and provable construction approach by using a special approximate plane satisfying the maximal angle condition described above. 

Let $\mathbb{Q}_1 = \text{Span}\{1,x,y,z,xy,xz,yz,xyz\}$ be the trilinear polynomial space and let $F$ be the centroid of the triangle described above and shown in Figure \ref{fig:subfig_interf}. According to \cite{2020GuoLin}, the local trilinear IFE space $S_h(T)$ is formed by piecewise polynomials $\phi_T$ with $\phi^{\pm}_T = \phi_T|_{T^{\pm}} \in \mathbb{Q}_1$ which satisfy the approximate jump conditions to \eqref{jump_cond_1} and \eqref{jump_cond_2}:
\begin{subequations}
\label{ife_shape_fun_2}
\begin{align}
\phi^{-}_T|_{\tau_T}=\phi^{+}_T|_{\tau_T}, ~~ \bfd(\phi^{-}_T) = \bfd(\phi^{+}_T), \label{ife_shape_fun_2_1} \\
\beta^-\nabla\phi^{-}_T(F)\cdot \bar{\mathbf{ n}} = \beta^+\nabla\phi^{+}_T(F)\cdot \bar{\mathbf{ n}},  \label{ife_shape_fun_2_2}
\end{align}
\end{subequations}
where $\bfd(p)$ is the vector of coefficients of terms $xy$, $yz$, $xz$ and $xyz$ in a polynomial $p\in \mathbb{Q}_1$. Then, on each interface element, we recall the extension operator $\mathcal{C}_T$ from (3.4) in \cite{2020GuoLin}
\begin{equation}
\begin{split}
\label{ext_opera}
\mathcal{C}_T:\mathbb{Q}_1\rightarrow\mathbb{Q}_1, ~~ & \text{such that} ~  \phi^-_T = p \in \mathbb{Q}_1  \text{~and~} \phi^+_T= \mathcal{C}_T(p) \in \mathbb{Q}_1\\
 &\text{together satisfy the approximate jump conditions \eqref{ife_shape_fun_2}}. 
\end{split}
\end{equation}
Let $L(X)=(X-F)\cdot\bar{\bfn}$ be the level-set function of the plane $\tau_T$. In particular, we have the following explicit expression for the operator:
\begin{subequations}
\label{C_rela}
\begin{align}
& \mathcal{C}_T(p) = p +  \left( \frac{\beta^-}{\beta^+}-1 \right) (\nabla p(F) \cdot \bar{\mathbf{ n}} ) L,\label{C_rela_1} \\
& \mathcal{C}^{-1}_T(p) = p +  \left( \frac{\beta^+}{\beta^-}-1 \right) (\nabla p(F) \cdot \bar{\mathbf{ n}} ) L.\label{C_rela_2}
\end{align}
\end{subequations}
Then the proposed IFE space $S_h(T)$ can be written as
\begin{equation}
\label{IFE_space_C}
S_h(T) = \{ \phi_T|_{T^\pm}\in\mathbb{Q}_1 : \phi_T|_{T^-}=p \in \mathbb{Q}_1 ~ \text{and} ~  \phi_T|_{T^+}= \mathcal{C}_T(p)  \},~~\forall T \in \mathcal{T}_h^i.
\end{equation}
It is crucial in both analysis and computation which shape functions are used, namely which degrees of freedom are chosen. Different shape functions may have different features in computation. In this article, we consider the Lagrange IFE shape function $\phi_{i,T}$ such that
\begin{equation}
\label{IFE_shape_fun_delta}
\phi_{i,T}(A_j) = \delta_{i,j}, ~~~ i,j=1,...,8,
\end{equation}
where $A_j$ are the vertices of the interface element $T$ as shown in Figure \ref{fig:subfig_interf}. Then the IFE space can be rewritten as
\begin{equation}
\label{loc_IFE_space_span}
S_h(T) = \text{Span} \{ \phi_{i,T}~:~ i=1,...,8 \}.
\end{equation}
The global IFE space is defined as
\begin{equation}
\label{glob_IFE_space}
S_h(\Omega) = \{ v\in L^2(\Omega): v|_T\in S_h(T)~ \forall T\in \mathcal{T}_h, ~ v~\text{is continuous at}~ X\in \mathcal{N}_h     \}.
\end{equation}
It has been shown in \cite{2020GuoLin} that these IFE spaces have optimal approximation capabilities to the functions satisfying the jump conditions in the $L^2$ and $H^1$ norms. We shall discuss some new approximation capabilities in Section \ref{sec:error}. Let $S_{h}^0(\Omega)$ be the subspace of $S_h(\Omega)$ with zero trace on $\partial\Omega$. Clearly, $S_{h}^0(\Omega)$ is a subspace of the underling space
\begin{equation}
\begin{split}
\label{vh_space}
V_h(\Omega) = \{ v\in L^2(\Omega)~:~ & v|_T\in H^1(T)~ \forall T\in \mathcal{T}^n_h, ~ \text{and} ~ v|_{T^\pm}\in H^1(T^\pm)~ \forall T\in \mathcal{T}^i_h, \\
& ~ v~\text{is continuous at each}~ X\in \mathcal{N}_h, ~ v|_{\partial\Omega}=0    \}.
 \end{split}
\end{equation}
Now the proposed PPIFE method is: find $u_h\in S_h(\Omega)$ such that $u_h(X)=g(X)$ $\forall X\in \mathcal{N}_h\cap\partial\Omega$ and
\begin{equation}
\label{ppife}
a_h(u_h,v_h) = L(v_h), ~~~ \forall v_h\in S_h^0(\Omega),
\end{equation}
where the bilinear form $a_h(\cdot,\cdot)$ is given by
\begin{equation}
\begin{split}
\label{ppife_1}
a_h(u,v) =& \sum_{T\in\mathcal{T}_h} \int_T \beta \nabla u \cdot \nabla v dX 
- \sum_{F\in\mathcal{F}^i_h} \int_F \aver{\beta \nabla u\cdot \mathbf{ n}} \jump{v} ds \\
  &   + \epsilon \sum_{F\in\mathcal{F}^i_h} \int_F \aver{\beta \nabla v\cdot \mathbf{ n}} \jump{u} ds + \sum_{F\in\mathcal{F}^i_h} \frac{\sigma }{h} \int_F \jump{u}\,\jump{v} ds,
\end{split}
\end{equation}
with $\sigma=\tilde{\sigma}^0(\beta^+)^2/\beta^-$ with $\tilde{\sigma}^0$ large enough but independent of $h$ and $\beta^{\pm}$, and the linear form $L: S_h(\Omega)\rightarrow \mathbb{R}$ is 
\begin{equation}
\label{ppife_2}
L(v) = \int_{\Omega} f v dX. 
\end{equation}
Here we emphasize that \eqref{ppife} is not a discontinuous Galerkin scheme since the global IFE functions in \eqref{glob_IFE_space} are all continuous at the mesh nodes such that the global degrees of freedom are isomorphic to the standard continuous piecewise trilinear finite element space. The penalties in \eqref{ppife_1} are only added on the interface faces of which the purpose is to handle the discontinuities of IFE functions across the element boundaries. This isomorphism makes the IFE method advantageous when solving moving interface problems \cite{2013LinLinZhang1}.

\section{Trace and Inverse Inequalities}
In this section, we proceed to establish the trace and inverse inequalities of the trilinear IFE functions. We note that these inequalities are non-trivial since IFE functions as piecewise polynomials do not have sufficient regularity for classical results to be applied. We need to consider all  interface element configurations in Figure \ref{fig:subfig_interf} separately. However, we note that their analysis is mathematically similar to each other; thus without loss of generality we only consider the Type III interface element as shown in Figure \ref{fig:subfig_interf}(\subref{inter_elem_case3}) since it is a good representative of our arguments.

We begin with a norm equivalence for polynomials on interface elements. For each interface element $T$, we denote the subelements cut by the interface $\Gamma$ by $T_1$ and $T_2$ where $T_1$ contains the vertex $A_1$ and $T_2$ contains the vertex $A_8$. Similarly, we have the subelements $\widetilde{T}_1$ and $\widetilde{T}_2$ cut by the approximating plane $\tau_T$. Then we have the following results.

\begin{lemma}
\label{norm_equiv}
On an interface element $T$, the following norm equivalence holds
\begin{equation}
\label{norm_equiv_01}
\|\cdot \|_{L^2(T_1)} \simeq \| \cdot \|_{L^2(\widetilde{T}_1)} \simeq \|\cdot\|_{L^2(T)}, ~~~~ \text{on} ~ \mathbb{Q}_1.
\end{equation} 
for the interface element types:
\begin{itemize}
\item Type III in Figure \ref{fig:subfig_interf}, if $|A_4D_4|\le\frac{1}{2}|A_4A_3|$ or $|A_2D_1|\le\frac{1}{2}|A_2A_1|$ or $|A_6D_2|\le\frac{1}{2}|A_6A_5|$ or $|A_8D_3|\le\frac{1}{2}|A_8A_7|$.
\item Type V in Figure \ref{fig:subfig_interf}.
\end{itemize}
In addition, the following norm equivalence holds
\begin{equation}
\label{norm_equiv_02}
\|\cdot \|_{L^2(T_2)} \simeq \| \cdot \|_{L^2(\widetilde{T}_2)} \simeq \|\cdot\|_{L^2(T)}, ~~~~ \text{on} ~ \mathbb{Q}_1.
\end{equation}
for the interface element types:
\begin{itemize}
\item Type I and Type II in Figure \ref{fig:subfig_interf}.
\item Type III in Figure \ref{fig:subfig_interf} if $|A_4D_4|\ge\frac{1}{2}|A_4A_3|$ or $|A_2D_1|\ge\frac{1}{2}|A_2A_1|$ or $|A_6D_2|\ge\frac{1}{2}|A_6A_5|$ or $|A_8D_3|\ge\frac{1}{2}|A_8A_7|$;
\item  Type IV and Type V in Figure \ref{fig:subfig_interf}.
\end{itemize}
\end{lemma}
\begin{proof}
Without loss of generality, we consider the Type III interface element in Figure \ref{fig:subfig_interf}(\subref{inter_elem_case3}), and only prove \eqref{norm_equiv_01}. In this case the approximate plane $\tau_T$ passing through any three points of $D_1$, $D_2$, $D_3$ and $D_4$ has the geometric approximation given in Theorem \ref{lem_interf_element_est}; and thus by symmetricity, without loss of generality, we only need to consider the case $|A_4D_4|\le\frac{1}{2}|A_4A_3|$. First of all, the hypothesis (\textbf{H2}) and \eqref{interf_element_est_eq0_dist} indicate that 
\begin{equation}
\label{norm_equiv_eq1}
\text{dist}(\tau_T,\Gamma\cap T) \le  12.0927(\kappa h)h \le 0.3386 h.
\end{equation}
We then consider the pyramid $A_3E_1E_2E_3$ denoted by $P_1$ with $E_1$, $E_2$ and $E_3$ satisfying $|A_3E_3|/|A_3A_7|=|A_3E_1|/|A_3A_1|=1/10$ and $|A_3E_2|/|A_3A_4|=1/20$ as shown in Figure \ref{fig:norm_equiv}. We can directly calculate that the shortest distance from $E_1$, $E_2$ and $E_3$ to the plane $\tau_T$ is 
$\sqrt{1/6}\cdot0.9 h=0.3674h>0.3386 h$. Hence, \eqref{norm_equiv_eq1} shows that $E_1,$ $E_2$ and $E_3$ are all in $T_1$, and thus the pyramid $A_3E_1E_2E_3$ is always inside $T_1$ regardless of the interface location. Therefore, we can construct a new pyramid $A_3E_1'E_2'E_3'$ denoted by $P'_1$ such that it is homothetic to $P_1$ and always contains the cubic element $T$, as illustrated by Figure \ref{fig:norm_equiv}. By Lemma 2.2 in \cite{2016WangXiaoXu}, we have for any $v\in \mathbb{Q}_1$, there holds
\begin{equation}
\label{norm_equiv_eq2} 
\| v \|_{L^2(T)} \le C \| v \|_{L^2(P'_1)} \le C \| v \|_{L^2(P_1)}  \le C \| v \|_{L^2(T_1)}. 
\end{equation}
Clearly, the pyramid $P_1$ must be always inside the subelement $\widetilde{T}_1$, then by similar derivation to \eqref{norm_equiv_eq2}, we still have $\| v \|_{L^2(T)}\le C \| v \|_{L^2(\widetilde{T}_1)}$. Using the simple geometry, we immediately have
\[\| v \|_{L^2(T_1)}\le C \| v \|_{L^2(T)},~~~\text{ and}~~~ \| v \|_{L^2(\widetilde{T}_1)}\le C \| v \|_{L^2(T)}.\] Combining these estimates, we arrive at \eqref{norm_equiv_01}.
\end{proof}
\begin{figure}[h]
\centering
     \includegraphics[width=1.8in]{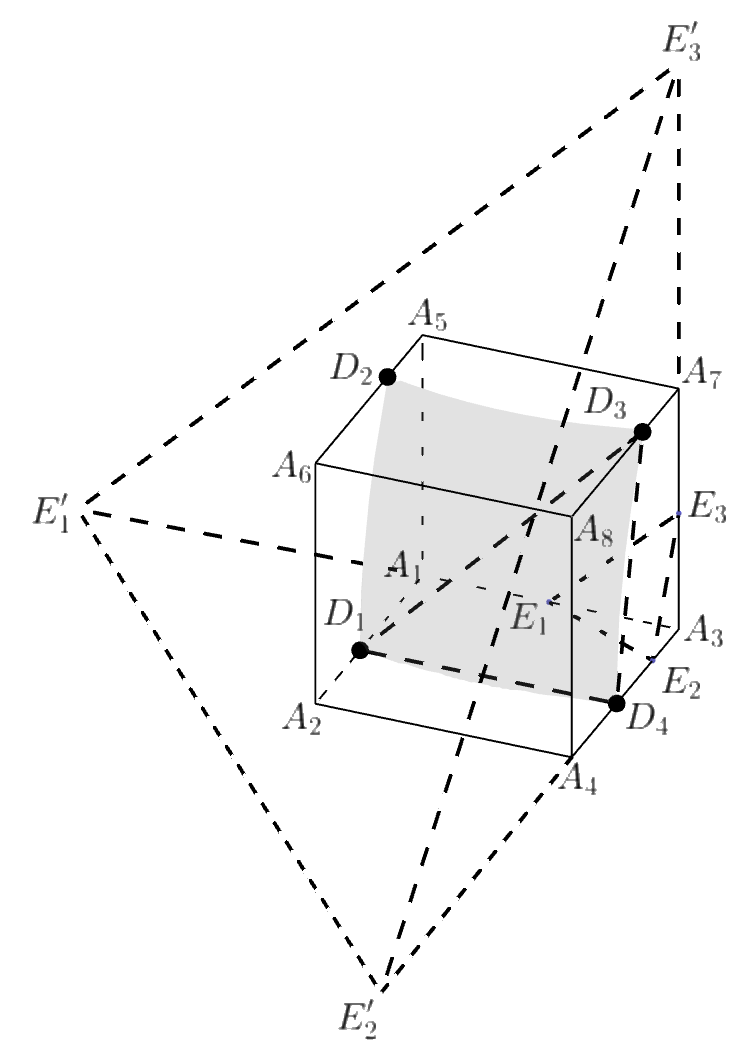}
     \caption{Type III interface element inclusion}
  \label{fig:norm_equiv} 
\end{figure}

Next we prove the following kind of trace inequality on a pyramid.

\begin{lemma}
\label{lem_plane_est}
Given a pyramid $K$ with a convex polygonal base $B$, suppose $B$ has $N_B$ edges, then
\begin{equation}
\label{lem_plane_est_eq0}
\|  p(X_0) \| \le C N^{1/2}_B  |K|^{-1/2} \| p \|_{L^2(K)}, ~~~ \forall p\in \mathbb{Q}_1, ~~\forall X_0\in B.
\end{equation}
\end{lemma}
\begin{proof}
We connect $X_0$ and the vertices of $B$, and thus obtain $N_B$ triangles denoted by $\triangle_i$, $i=1,\cdots,N_B$. Then we connect $X_0$ and the apex of the pyramid to obtain $N_B$ sub-pyramids denoted by $K_i$, $i=1,\cdots,N_B$.
Without loss of generality, we assume $|K_1|\ge |K_2|\ge\cdots\ge |K_{N_B}|$. Then $|K| = \sum_{i=1}^{N_B} |K_i| \le N_B |K_1|$.
Thus, on $K_1$, the standard trace inequality for polynomials \cite{2003WarburtonHesthaven} yields
\begin{equation}
\label{lem_plane_est_eq1}
\begin{split}
|p(X_0)|&\le C |\triangle_1|^{-1/2} \| p \|_{L^2(\triangle_1)} \le C |\triangle_1|^{-1/2} \left( \frac{|\triangle_1|}{|K_1|} \right)^{1/2} \| p \|_{L^2(K_1)} \le C N^{1/2}_B  |K|^{-1/2} \| p \|_{L^2(K)}.
 \end{split}
\end{equation}
\end{proof}

Let $T$ be an interface element of the configuration shown in Figure \ref{fig:subfig_interf}. Recall that the subelement $T_1$ contains the vertex $A_1$ and $T_2$ contains $A_8$. Then we have the following stability estimates for $\mathcal{C}_T$.

\begin{lemma}
\label{lem_C_stab}
On each interface element $T$, there holds
\begin{subequations}
\label{lem_C_stab_eq0}
\begin{align}
& |\mathcal{C}_T(p)|_{H^j(T_2)} \le C |p|_{H^j(T_2)} , ~~~~ j=0,1, ~~ \forall p\in \mathbb{Q}_1, \label{lem_C_stab_eq01} \\
& |\mathcal{C}^{-1}_T(p)|_{H^j(T_2)} \le C \frac{\beta^+}{\beta^-}   |p|_{H^j(T_2)} ,  ~~~~ j=0,1, ~~\forall p\in \mathbb{Q}_1, \label{lem_C_stab_eq02} 
\end{align}
\end{subequations}
for the interface element types:
\begin{itemize}
\item Types I and II in Figure \ref{fig:subfig_interf};
\item Type III in Figure \ref{fig:subfig_interf}, if $|A_4D_4|\ge\frac{1}{2}|A_4A_3|$ or $|A_2D_1|\ge\frac{1}{2}|A_2A_1|$ or $|A_6D_2|\ge\frac{1}{2}|A_6A_5|$ or $|A_8D_3|\ge\frac{1}{2}|A_8A_7|$;

\item Types IV and V in Figure \ref{fig:subfig_interf}.
\end{itemize}
\end{lemma}
\begin{proof}
Also we only show the proof for the interface element of Type III in Figure \ref{fig:subfig_interf}. First of all, we note that
\begin{equation}
\label{lem_C_stab_eq1}
|L|_{H^j(T_2)}  \le C h^{1-j} |T_2|^{1/2} \le C h^{5/2-j}, ~~~ j=0,1.
\end{equation}
Due to symmetry, we only need to consider the edge $A_4A_3$, namely assuming $|A_4D_4|\ge \frac{1}{2}|A_4A_3|$ as shown in Figure \ref{fig:interface_elem_stability}(\subref{inter_elem_case3_stability_1}). We consider the tetrahedron $A_4D_1D_3D_4$ denoted as $P$. Since $|A_4D_4|\ge \frac{1}{2}|A_4A_3|$, we know that $|P| \ge  h^3/12$. Therefore, according to Lemma \ref{lem_plane_est} and \eqref{lem_C_stab_eq1}, we use \eqref{C_rela_1} to obtain
\begin{equation}
\begin{split}
\label{lem_C_stab_eq2}
| \mathcal{C}_T(p) |_{H^j(T_2)}  & \le | p |_{H^j(T_2)} +  \max\Big\{ \frac{\beta^-}{\beta^+},1\Big\} \|\nabla p(F) \| |L|_{H^j(T_2)} \\
& \le | p |_{H^j(T_2)} + C h^{1-j} \| \nabla p \|_{L^2(P)}   \le  | p |_{H^j(T_2)} + C  | p |_{H^j(P)}
\end{split}
\end{equation}
where in the last inequality we have also used the inverse inequality for $j=0$ on $P$. Furthermore, recalling that $\widetilde{T}_2$ is the subelement cut by the plane passing through $D_1D_4D_3$, and applying \eqref{norm_equiv_02} to the last inequality in \eqref{lem_C_stab_eq2}, we have
\begin{equation}
\begin{split}
\label{lem_C_stab_eq3}
| \mathcal{C}_T(p) |_{H^j(T_2)} & \le  | p |_{H^j(T_2)} + C | p |_{H^j(\widetilde{T}_2)}  \le  | p |_{H^j(T_2)} + C  | p |_{H^j(T_2)},
\end{split}
\end{equation}
which gives \eqref{lem_C_stab_eq01}. For \eqref{lem_C_stab_eq02}, similar to \eqref{lem_C_stab_eq2} and \eqref{lem_C_stab_eq3}, we use \eqref{C_rela_2} to obtain
 \begin{equation}
\begin{split}
\label{lem_C_stab_eq4}
| \mathcal{C}^{-1}_T(p) |_{H^j(T_2)}  & \le | p |_{H^j(T_2)} +  \max\Big\{ \frac{\beta^+}{\beta^-},1\Big\} \|\nabla p(F) \| |L|_{H^j(T_2)}  \le  | p |_{H^j(T_2)} + C \frac{\beta^+}{\beta^-} | p |_{H^j(P)} \le C \frac{\beta^+}{\beta^-} | p |_{H^j(T_2)},
\end{split}
\end{equation}
which finishes the proof.
\end{proof}

\begin{figure}[h]
\centering
\begin{subfigure}{.4\textwidth}
     \includegraphics[width=1.8in]{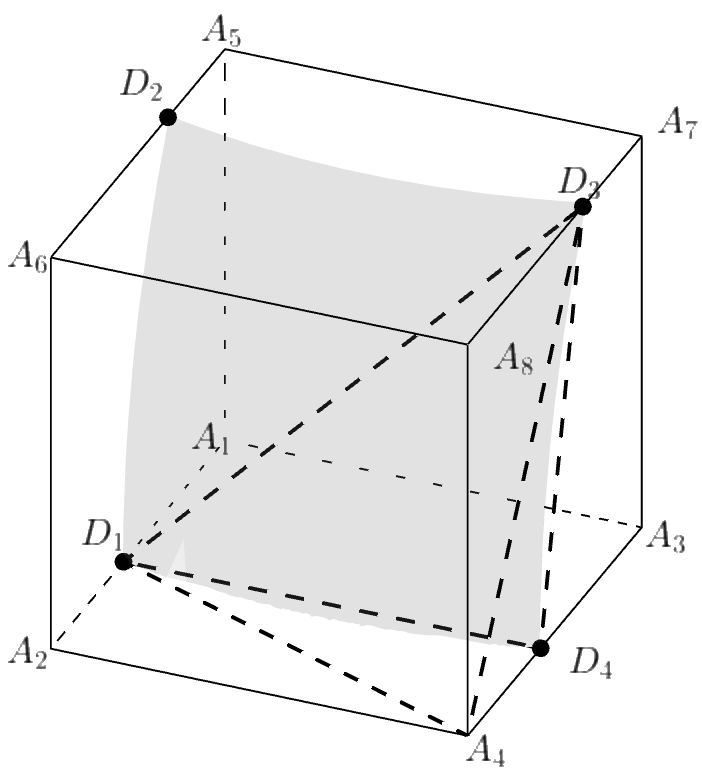}
     \caption{If $|A_4D_4|\ge \frac{1}{2}|A_4A_3|$ }
     \label{inter_elem_case3_stability_1} 
\end{subfigure}
\begin{subfigure}{.4\textwidth}
     \includegraphics[width=1.8in]{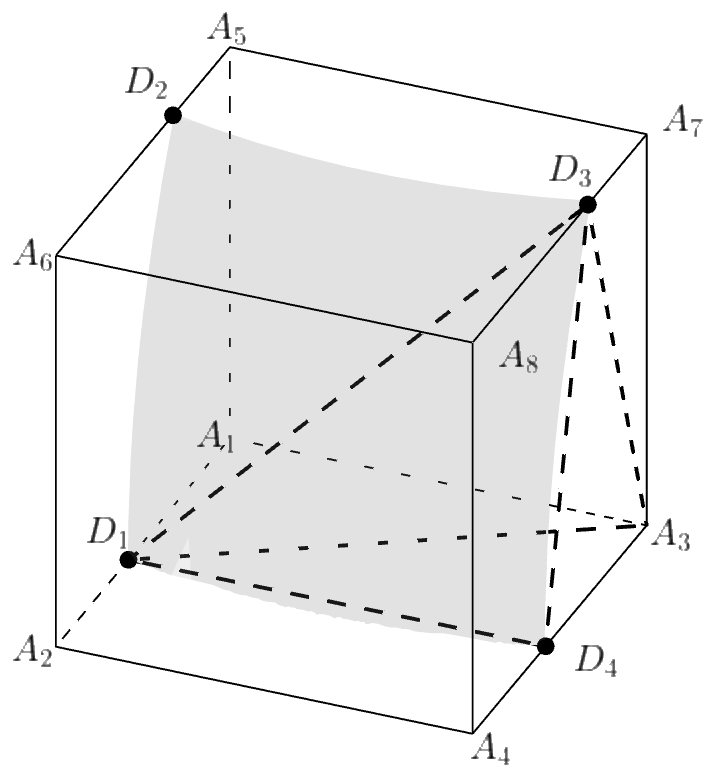}
     \caption{If $|A_4D_4| \le \frac{1}{2}|A_4A_3|$ }
     \label{inter_elem_case3_stability_2} 
\end{subfigure}
     \caption{Proof of Lemmas \ref{lem_C_stab} and \ref{lem_C_stab2}.}
  \label{fig:interface_elem_stability} 
\end{figure}

\begin{lemma}
\label{lem_C_stab2}
On each interface element $T$, there holds
\begin{subequations}
\label{lem_C_stab2_eq0}
\begin{align}
& |\mathcal{C}_T(p)|_{H^j(T_1)} \le  C |p|_{H^j(T_1)} , ~~~~ j=0,1, ~~ \forall p\in \mathbb{Q}_1, \label{lem_C_stab2_eq01}  \\
& |\mathcal{C}^{-1}_T(p)|_{H^j(T_1)} \le  C \frac{\beta^+}{\beta^-}   |p|_{H^j(T_1)} ,  ~~~~ j=0,1, ~~\forall p\in \mathbb{Q}_1, \label{lem_C_stab2_eq02} 
\end{align}
\end{subequations}
for the interface element types
\begin{itemize}
\item Type III in Figure \ref{fig:subfig_interf}, if $|A_4D_4|\le \frac{1}{2}|A_4A_3|$ or $|A_2D_1|\le \frac{1}{2}|A_2A_1|$ or $|A_6D_2|\le \frac{1}{2}|A_6A_5|$ or $|A_8D_3|\le \frac{1}{2}|A_8A_7|$;
\item Type V in Figure \ref{fig:subfig_interf}.
\end{itemize}
\end{lemma}
\begin{proof}
We still only consider the interface element of Type III and the edge $A_4A_3$, and without loss of generality we assume $|A_4D_4|\le \frac{1}{2}|A_4A_3|$, i.e., $|A_3D_4|\ge |A_4A_3|/2$ as shown in Figure \ref{fig:interface_elem_stability}(\subref{inter_elem_case3_stability_2}). In this case, we consider the tetrahedron $P=A_3D_1D_4D_3$ and by a similar discussion, we also have $|P| \ge  h^3/12$. Therefore, similar to \eqref{lem_C_stab_eq2}, we have
\begin{equation}
\begin{split}
\label{lem_C_stab2_eq1}
| \mathcal{C}_T(p) |_{H^j(T_1)}  & \le  | p |_{H^j(T_1)} +  \max\Big\{ \frac{\beta^-}{\beta^+},1\Big\} \|\nabla p(F) \| |L|_{H^j(T_1)} \\
& \le  | p |_{H^j(T_1)} + C h^{1-j}  \| \nabla p \|_{L^2(P)}  \le   | p |_{H^j(T_1)} + C  | p |_{H^j(P)}
\end{split}
\end{equation}
where in the last inequality we also use the inverse inequality for $j=0$ on $P$. Finally, similar to \eqref{lem_C_stab_eq3} but applying \eqref{norm_equiv_01} to bound the last term in \eqref{lem_C_stab2_eq1}, we have \eqref{lem_C_stab2_eq01}. \eqref{lem_C_stab2_eq02} can be proved by a similar argument.
\end{proof}

The estimates above can be understood as the stability of the extension operator $\mathcal{C}_T$ for polynomials, and they serve as the foundation of the stability of the PPIFE method, namely, the inverse and trace inequalities. Now we are ready to present those inequalities for the proposed IFE functions. Here we emphasize that both of these inequalities rely essentially on the stability of the operator $\mathcal{C}_T$ given by Lemmas \ref{lem_C_stab} and \ref{lem_C_stab2}.
\begin{theorem}[Trace Inequalities]
\label{thm_trace_inequa}
On each interface element $T$ and its face $F$, one of the following must hold
\begin{subequations}
\label{thm_trace_inequa_eq0}
\begin{align}
    &  \| \nabla\phi_T\cdot\mathbf{ n} \|_{L^2(F)} \le  C h^{-1/2} \| \nabla \phi \|_{L^2(T)}, \label{thm_trace_inequa_eq01}  \\
    &  \| \beta \nabla\phi_T\cdot\mathbf{ n} \|_{L^2(F)} \le  C h^{-1/2} \| \beta \nabla \phi \|_{L^2(T)}.  \label{thm_trace_inequa_eq02}
\end{align}
\end{subequations}
\end{theorem}
\begin{proof}
We also only present the detailed proof for the interface element of Type III in Figure \ref{fig:subfig_interf}. Due to the symmetry, we can assume the subelement containing $A_1$ is $T^-$, i.e., $T_1=T^-$, and then the subelement containing $A_8$ is $T^+$, i.e., $T_2=T^+$. Furthermore, without loss of generality, we only consider the interface face $F=A_1A_2A_3A_4$. Here we note that $F^s=F\cap T^s$, $s=\pm$, are all curved-edge quadrilaterals. According to the definition \eqref{IFE_space_C}, for each IFE function $\phi_T$ there exists a polynomial $p\in\mathbb{Q}_1$ such that $\phi_T=\phi^-_T=p$ on $T^-$ and $\phi_T=\phi^+_T=\mathcal{C}_T(p)$ on $T^+$. 

On one hand, we first consider the case $|A_4D_4|\ge  \frac{1}{2}|A_4A_3|$. On $T^+$, we simply apply the standard trace inequality \cite{2003WarburtonHesthaven} on the whole element $T$ to obtain
\begin{equation}
\begin{split}
\label{thm_trace_inequa_eq1}
\| \beta^+ \nabla\phi_T\cdot\mathbf{ n}\|_{L^2(F^+)} &= \| \beta^+ \nabla \mathcal{C}_T(p)\cdot\mathbf{ n}\|_{L^2(F^+)} \le  \| \beta^+ \nabla \mathcal{C}_T(p)\cdot\mathbf{ n}\|_{L^2(F)}  \\
 & \le  Ch^{-1/2} |\beta^+\mathcal{C}_T(p)|_{H^1(T)} \le   Ch^{-1/2} |\beta^+\mathcal{C}_T(p)|_{H^1(T^+)} 
\end{split}
\end{equation} 
where in the last inequality we have used \eqref{norm_equiv_02}. Similarly, applying the standard trace inequality \cite{2003WarburtonHesthaven} on the whole element $T$ with \eqref{lem_C_stab2_eq02}, we have
\begin{equation}
\begin{split}
\label{thm_trace_inequa_eq2}
 &\| \nabla\phi_T\cdot\mathbf{ n}\|_{L^2(F^-)} = \| \nabla p \cdot\mathbf{ n}\|_{L^2(F^-)} \le  \| \nabla p \cdot\mathbf{ n}\|_{L^2(F)} \le  Ch^{-1/2} |p|_{H^1(T)} \\
\le &  Ch^{-1/2} \left( |\mathcal{C}^{-1}_T(\mathcal{C}_T(p))|_{H^1(T^+)} + |p|_{H^1(T^-)} \right) \le  Ch^{-1/2} \left( \frac{\beta^+}{\beta^-} |\mathcal{C}_T(p)|_{H^1(T^+)} + |p|_{H^1(T^-)} \right). 
\end{split}
\end{equation}
Combining \eqref{thm_trace_inequa_eq1} and \eqref{thm_trace_inequa_eq2}, we have the desired result \eqref{thm_trace_inequa_eq02}.

On the other hand, if $|A_4D_4|\le  \frac{1}{2}|A_4A_3|$, we apply the argument \eqref{thm_trace_inequa_eq1} to $\nabla\phi_T\cdot\mathbf{ n}$ on $T^-$ with \eqref{norm_equiv_01} to obtain
\begin{equation}
\begin{split}
\label{thm_trace_inequa_eq3}
\| \nabla\phi_T\cdot\mathbf{ n}\|_{L^2(F^-)} &= \| \nabla p\cdot\mathbf{ n}\|_{L^2(F^-)} \le  \| \nabla p\cdot\mathbf{ n}\|_{L^2(F)} \le  Ch^{-1/2} | p |_{H^1(T)} \le  Ch^{-1/2} | p |_{H^1(T^-)}. 
\end{split}
\end{equation}
In addition, we apply the argument \eqref{thm_trace_inequa_eq2} to $\nabla\phi_T\cdot\mathbf{ n}$ on $T^+$ with \eqref{lem_C_stab_eq01} to obtain
\begin{equation}
\begin{split}
\label{thm_trace_inequa_eq4}
& \| \nabla\phi_T\cdot\mathbf{ n}\|_{L^2(F^+)} = \| \nabla \mathcal{C}_T(p) \cdot\mathbf{ n}\|_{L^2(F^+)} \le  \| \nabla \mathcal{C}_T(p) \cdot\mathbf{ n}\|_{L^2(F)}  \le  Ch^{-1/2} |\mathcal{C}_T(p)|_{H^1(T)}\\
 \le &  Ch^{-1/2} \left( |\mathcal{C}_T(p)|_{H^1(T^+)} + | \mathcal{C}_T(p) |_{H^1(T^-)} \right) \le  Ch^{-1/2} \left( |\mathcal{C}_T(p)|_{H^1(T^+)} + |p|_{H^1(T^-)} \right) \le  Ch^{-1/2} |\phi_T|_{H^1(T)}.
\end{split}
\end{equation}
Finally combining \eqref{thm_trace_inequa_eq3} and \eqref{thm_trace_inequa_eq4}, we have \eqref{thm_trace_inequa_eq01}
\end{proof}

\begin{remark}
\label{rem_trace_inequa}
Roughly speaking, for each interface element $T$, if the size of the subelement corresponding the larger coefficient $\beta^+$ shrinks to $0$, then the trace inequality \eqref{thm_trace_inequa_eq02} holds. On the other hand, if the subelement corresponding the smaller coefficient $\beta^-$ shrinks, then \eqref{thm_trace_inequa_eq01} holds. These two inequalities can be unified as the following one
\begin{equation}
\label{rem_trace_inequa_eq}
\| \beta \nabla\phi_T\cdot\mathbf{ n} \|_{L^2(F)} \le  C h^{-1/2} \frac{\beta^+}{\sqrt{\beta^-}} \| \sqrt{\beta} \nabla \phi \|_{L^2(T)}.
\end{equation}
\end{remark}

In addition, we note that the IFE functions on interface elements may not be continuous across the interface. Here we present a special type of trace inequality for IFE functions which shows the difference between the two polynomial components on interface can be bounded by the IFE function on the element with certain optimal order of $h$. This is important to estimate the inconsistence error. For this purpose, let us first estimate $L(X)=(X-F)\cdot\bar{\mathbf{ n}}$ on $\Gamma\cap T$.

\begin{lemma}
\label{lem_L_est}
For each interface element $T$, there holds
\begin{equation}
\label{lem_L_est_eq0}
\| L \|_{L^2(\Gamma\cap T)} \le  Ch^3.
\end{equation}
\end{lemma}
\begin{proof}
For each $X\in\Gamma\cap T$, we denote the projection of $X$ onto the approximate plane $\tau_T$ by $X_{\bot}$. Using \eqref{interf_element_est_eq0_dist}, the fact $L(X)=0$ for every $X\in\tau_T$, and Lemma \ref{lem_interf_area} we have
\begin{equation}
\begin{split}
\label{lem_L_est_eq1}
\| L \|_{L^2(\Gamma\cap T)} &= \left ( \int_{\Gamma\cap T} ((X-F)\cdot\bar{\mathbf{ n}})^2 dS  \right)^{1/2} = \left ( \int_{\Gamma\cap T} ((X-X_{\bot})\cdot\bar{\mathbf{ n}})^2 dS  \right)^{1/2} \le  C \| X - X_{\bot}\| |\Gamma\cap T|^{1/2} \le  Ch^3.
\end{split}
\end{equation}
\end{proof}

\begin{theorem}
\label{thm_trace_interface}
For each interface element $T$, there holds
\begin{equation}
\label{thm_trace_interface_eq0}
\| [\phi_T] \|_{L^2(T\cap\Gamma)} \le  C \frac{\sqrt{\beta^+}}{\beta^-} h^{3/2} \| \sqrt{\beta} \nabla \phi_T \|_{L^2(T)}, ~~~ \forall \phi_T \in S_h(T).
\end{equation}
\end{theorem}
\begin{proof}
We only consider the interface element of Type III shown in Figure \ref{fig:subfig_interf}(\subref{inter_elem_case3}), and without loss of generality we assume that the subelement $T_1$ containing $A_1$ is $T^-$, i.e., $T_1=T^-$, and then the subelement $T_2$ containing $A_8$ is $T^+$, i.e., $T_2=T^+$. According to the relation $\mathcal{C}_T$ between the two polynomial components of an IFE function \eqref{C_rela}, we note that there exists a polynomial $p\in\mathbb{Q}_1$ such that $\phi^+_T=\mathcal{C}_T(p)$, $\phi^-_T=p$ and
\begin{equation}
\label{thm_trace_interface_eq1}
 [\phi_T] = \left( \frac{\beta^-}{\beta^+}-1 \right) (\nabla p(F) \cdot \bar{\mathbf{ n}} ) L = \left( \frac{\beta^+}{\beta^-}-1 \right) (\nabla \mathcal{C}_T(p)(F) \cdot \bar{\mathbf{ n}} ) L.
\end{equation}

If $|A_4D_4|\ge \frac{1}{2}|A_4A_3|$, using the similar argument to \eqref{lem_C_stab_eq2} and \eqref{lem_C_stab_eq3} with Lemma \ref{lem_plane_est} on the tetrahedron $A_4D_1D_4D_3$ as shown in Figure \ref{fig:interface_elem_stability}(\subref{inter_elem_case3_stability_1}), we have 
$|\nabla \mathcal{C}_T(p)(F) \cdot \bar{\mathbf{ n}}|\le  Ch^{-3/2}\|\nabla \mathcal{C}_T(p)\|_{L^2(T^+)}$.
Then we use the second equality in \eqref{thm_trace_interface_eq1} and Lemma \ref{lem_L_est} to obtain
\begin{equation}
\begin{split}
\label{thm_trace_interface_eq2}
\| [\phi_T] \|_{L^2(\Gamma\cap T)} & \le  C\frac{\beta^+}{\beta^-} |\nabla \mathcal{C}_T(p)(F) \cdot \bar{\mathbf{ n}}| \| L \|_{L^2(\Gamma\cap T)} \\
& \le  C\frac{\beta^+}{\beta^-} h^{3/2} \|\nabla \mathcal{C}_T(p)\|_{L^2(T^+)} 
\le   C\frac{\sqrt{\beta^+}}{\beta^-} h^{3/2} \| \sqrt{\beta} \nabla \phi_T \|_{L^2(T)}.
\end{split}
\end{equation}
If $|A_4D_4|\le \frac{1}{2}|A_4A_3|$, using the similar argument to \eqref{lem_C_stab2_eq1} with Lemma \ref{lem_plane_est} on the tetrahedron $A_3D_1D_4D_3$ as shown in Figure \ref{fig:interface_elem_stability}(\subref{inter_elem_case3_stability_2}), we have
\[|\nabla p(F) \cdot \bar{\mathbf{ n}}|\le  Ch^{-3/2}\|\nabla  p\|_{L^2(T^-)}.\] 
Then we apply the first equality in \eqref{thm_trace_interface_eq1} and Lemma \ref{lem_L_est} to obtain
\begin{equation}
\begin{split}
\label{thm_trace_interface_eq3}
\| [\phi_T] \|_{L^2(\Gamma\cap T)} & \le  C\frac{\beta^-}{\beta^+} |\nabla p(F) \cdot \bar{\mathbf{ n}}| \| L \|_{L^2(\Gamma\cap T)}  \le  C \frac{\sqrt{\beta^-}}{\beta^+} h^{3/2} \| \sqrt{\beta^-}\nabla p \|_{L^2(T^-)} \le   C \frac{\sqrt{\beta^-}}{\beta^+} h^{3/2} \| \sqrt{\beta} \nabla \phi_T \|_{L^2(T)}.
\end{split}
\end{equation}
Combining \eqref{thm_trace_interface_eq2} and \eqref{thm_trace_interface_eq3} and noticing that $\frac{\sqrt{\beta^-}}{\beta^+}\le \frac{\sqrt{\beta^+}}{\beta^-}$, we have finished the proof.
\end{proof}

We now present an inverse inequality on the surface.
\begin{lemma}
\label{lem_inver_inequa_surf}
For each interface element $T$, there holds
\begin{equation}
| v |_{H^j(T\cap\Gamma)} \le  C h^{-j} \| v \|_{L^2(T\cap\Gamma)}, ~~~ \forall v \in \mathbb{Q}_1.
\end{equation}
\end{lemma}
\begin{proof}
For each $X\in \Gamma\cap T$, let $X_{\bot}$ be the projection of $X$ onto the plane $\tau_T$. Then by \eqref{interf_element_est_prod}, we have the following norm equivalence for $v\in H^1(T)$:
\begin{equation}
\label{thm_interface_interface_eq1}
\| v \|_{L^2(\Gamma\cap T)} = \left( \int_{\Gamma\cap T} v^2 dX \right)^{1/2} = \left( \int_{\tau_T} v^2 \frac{1}{\bar{\bfn}\cdot\bfn(X)} dX_{\bot} \right)^{1/2} \simeq \| v \|_{L^2(\tau_T)}.
\end{equation}
Hence, by the inverse inequality on the plane $\tau_T$, we have
\begin{equation}
\label{thm_interface_interface_eq2}
|v|_{H^j(\Gamma\cap T)} \le  C |v|_{H^j(\tau_T)} \le  C h^{-j} \|v\|_{L^2(\tau_T)} \le  C h^{-j} \|v\|_{L^2(\Gamma \cap T)}.
\end{equation}
\end{proof}

\begin{theorem}
\label{thm_interface_interface}
For each interface element $T$, there holds
\begin{equation}
\label{thm_interface_interface_eq0}
\| [\phi_T] \|_{L^2(T\cap\Gamma)} \le  C h^{2} | [\phi_T] |_{H^1(T\cap\Gamma)}, ~~~ \forall \phi_T \in S_h(T).
\end{equation}
\end{theorem}
\begin{proof}
For simplicity, we denote $w=[\phi_T]$ and note that $w=0$ on the approximate plane $\tau_T$. For each $X\in \Gamma\cap T$, let $X_{\bot}$ be the projection of $X$ onto $\tau_T$. Then the Taylor expansion yields
\begin{equation}
0=w(X_{\bot}) = w(X) + \partial_{\bfzeta}w(X) |X-X_{\bot}| +  \partial^2_{\bfzeta}w(X) |X-X_{\bot}|^2 + \partial^3_{\bfzeta}w(X) |X-X_{\bot}|^3 
\end{equation}
 where $\bfzeta$ is the directional vector from $X_{\bot}$ to $X$. Hence using \eqref{interf_element_est_eq0_dist} we have
 \begin{equation}
\| w \|_{L^2(\Gamma\cap T)} \le  C\Big( h^2|w|_{H^1(\Gamma\cap T)} + h^4|w|_{H^2(\Gamma\cap T)} + h^6|w|_{H^3(\Gamma\cap T)}\Big)
\end{equation}
which yields the desired result by Lemma \ref{lem_inver_inequa_surf}.
\end{proof}

\begin{theorem}[Inverse inequalities]
\label{thm_inver_inequa}
For each interface element $T$, there holds
\begin{equation}
\label{thm_inver_inequa_eq0}
\| \nabla\phi_T \|_{L^2(T)} \le  C \frac{\beta^+}{\beta^-} h^{-1} \| \phi \|_{L^2(T)}, ~~~~ \forall \phi_T\in S_h(T).
\end{equation}
\end{theorem}
\begin{proof}
Following the convention above, we again only discuss the interface element of Type III shown in Figure \ref{fig:subfig_interf}(\subref{inter_elem_case3}), and without loss of generality we assume that the subelement $T_1$ containing $A_1$ is $T^-$ while the subelement $T_2$ containing $A_8$ is $T^+$. Recall that there is a polynomial $p\in\mathbb{Q}_1$, $\phi^-_T=p$ and $\phi^+_T=\mathcal{C}_T(p)$. The argument is actually similar to the one for Theorem \ref{thm_trace_inequa}. 

First, if $|A_4D_4|\ge \frac{1}{2}|A_4A_3|$, then for $\nabla\phi_T$ on $T^+$, we apply the standard inverse inequality with \eqref{norm_equiv_02} to obtain
\begin{equation}
\begin{split}
\label{thm_inver_inequa_eq1}
\| \nabla\phi_T \|_{L^2(T^+)} &= \| \nabla \mathcal{C}_T(p) \|_{L^2(T^+)}  \le  \| \nabla \mathcal{C}_T(p) \|_{L^2(T)} \le  Ch^{-1} \| \mathcal{C}_T(p) \|_{L^2(T)}  \le  Ch^{-1} \| \mathcal{C}_T(p) \|_{L^2(T^+)}.
\end{split}
\end{equation}
For $\nabla\phi_T$ on $T^-$, we apply the standard inverse inequality and \eqref{lem_C_stab_eq02} to have
\begin{equation}
\begin{split}
\label{thm_inver_inequa_eq2}
\| \nabla \phi_T \|_{L^2(T^-)} &\le  \| \nabla p \|_{L^2(T)} \le  Ch^{-1} \| p \|_{L^2(T)} \le  Ch^{-1} \left( \| p \|_{L^2(T^-)}  + \| p \|_{L^2(T^+)}  \right)\\
&=  Ch^{-1} \left( \| p \|_{L^2(T^-)}  + \| \mathcal{C}^{-1}_T(\mathcal{C}_T(p)) \|_{L^2(T^+)}  \right) \le   Ch^{-1} \left( \| p \|_{L^2(T^-)}  + \frac{\beta^+}{\beta^-}\| \mathcal{C}_T(p) \|_{L^2(T^+)}  \right).
\end{split}
\end{equation}
Combining \eqref{thm_inver_inequa_eq1} and \eqref{thm_inver_inequa_eq2}, we have \eqref{thm_inver_inequa_eq0}.

Second, if $|A_4D_4|\le \frac{1}{2}|A_4A_3|$, then for $\nabla\phi_T$ on $T^+$, applying the argument in \eqref{thm_inver_inequa_eq2} but with \eqref{lem_C_stab2_eq01}, we obtain
\begin{equation}
\begin{split}
\label{thm_inver_inequa_eq3}
& \| \nabla \phi_T \|_{L^2(T^+)}  = \|\mathcal{C}_T(p) \|_{L^2(T^+)} \le  \| \nabla \mathcal{C}_T(p) \|_{L^2(T)} \le  Ch^{-1} \| \mathcal{C}_T(p)  \|_{L^2(T)}\\
 \le &  Ch^{-1} \left( \| \mathcal{C}_T(p) \|_{L^2(T^-)}  + \| \mathcal{C}_T(p) \|_{L^2(T^+)}  \right)  \le   Ch^{-1} \left( \| p \|_{L^2(T^-)}  + \| \mathcal{C}_T(p) \|_{L^2(T^+)}  \right).
\end{split}
\end{equation}
In addition, for $\nabla\phi_T$ on $T^-$, applying the argument in \eqref{thm_inver_inequa_eq1} but with \eqref{norm_equiv_01}, we have
\begin{equation}
\label{thm_inver_inequa_eq4}
\| \nabla\phi_T \|_{L^2(T^-)} = \| \nabla p \|_{L^2(T^-)}  \le  \| \nabla p \|_{L^2(T)} \le  Ch^{-1} \| p \|_{L^2(T)}  \le  Ch^{-1} \| p \|_{L^2(T^-)}.
\end{equation}
Combining \eqref{thm_inver_inequa_eq3} and \eqref{thm_inver_inequa_eq4}, we finish the proof.
\end{proof}


\section{Error Estimates of IFE Solutions}
\label{sec:error}

In this section, we proceed to estimate the errors of the PPIFE scheme \eqref{ppife}. For this purpose, we define the energy norm:
\begin{equation}
\label{energy_norm}
\tbar v \tbar^2 := \sum_{T\in\mathcal{T}_h} \| \sqrt{\beta} \nabla v\|^2_{L^2(T)} + \sum_{F\in\mathcal{F}^i_h } \sigma \| h^{-1/2} \jump{v} \|^2_{L^2(F)} + \sum_{F\in \mathcal{F}^i_h } \frac{1}{\sigma} \| h^{1/2} \aver{\beta\nabla v\cdot\mathbf{n}} \|^2_{L^2(F)}. 
\end{equation}
It is easy to see it is a semi-norm. We begin by showing that the functional above is indeed a norm on the space $V_h(\Omega)$

\begin{lemma}
\label{lem_full_norm}
$\tbar \cdot \tbar$ is a norm of $V_h(\Omega)$.
\end{lemma}
\begin{proof}
Since $\tbar v \tbar^2=0$, we directly have $\|\nabla v\|=0$, and thus $v$ is a constant on each element. Due to the continuity at mesh nodes and zero trace on $\partial\Omega$, we know that $v$ must be zero on the whole domain.
\end{proof}

Now we show that the bilinear form $a_h(\cdot,\cdot)$ is both continuous and coercive under the energy norm $\tbar \cdot \tbar$. 
\begin{theorem}
\label{thm_bound}
There exists a constant $C$ such that
\begin{equation}
\label{thm_bound_0}
a_h(v,w) \le  C \tbar v \tbar \tbar w \tbar, ~~~~ \forall v,w \in V_h(\Omega).
\end{equation}
\end{theorem}
\begin{proof}
It directly follows from the H\"older's inequality.
\end{proof}

\begin{theorem}
\label{thm_coer}
Assume $\sigma$ is large enough, then there holds
\begin{equation}
\label{thm_coer_eq0}
a_h(v,v) \ge  \frac{1}{4} \tbar v \tbar^2, ~~~~ \forall v\in S_h(\Omega). 
\end{equation}
\end{theorem}
\begin{proof}
We first note that
\begin{equation}
\begin{split}
\label{thm_coer_eq1}
a_h(v,v) &= \sum_{T\in\mathcal{T}_h} \| \sqrt{\beta} \nabla v \|^2_{L^2(T)} + (\epsilon - 1) \sum_{F\in \mathcal{F}^i_h} \int_F \{ \beta \nabla v\cdot \mathbf{ n} \} [v] ds + \sum_{F\in \mathcal{F}^i_h} \frac{\sigma}{h}  \| [v] \|^2_{L^2(F)} ds.
\end{split}
\end{equation}
Then we only need to bound the second term in \eqref{thm_coer_eq1}. 
On each interface face $F$, we denote its two neighbor elements by $T^1_F$ and $T^2_F$. Then we apply \eqref{rem_trace_inequa_eq} to obtain
\begin{equation}
\label{thm_coer_eq4}
\| \aver{ \beta \nabla v\cdot \mathbf{ n} } \|_{L^2(F)}
\le  \frac{1}{2} \sum_{j=1,2} \| \beta \nabla v|_{T^i_F}\cdot\mathbf{ n} \|_{L^2(F)} 
\le  C\frac{\beta^+}{2\sqrt{\beta^-}} \sum_{j=1,2} h^{-1/2} \| \sqrt{\beta} \nabla v \|_{L^2(T^j_F)}.
\end{equation}
Using H\"older's inequality and Young's inequality, we have
\begin{equation}
\begin{split}
\label{thm_coer_eq5}
\verti{ (\epsilon - 1)   \int_F \aver{ \beta \nabla v\cdot \mathbf{ n} } \jump{v} ds}  
\le  & 2  \left( h^{1/2}  \| \aver{ \beta \nabla v\cdot \mathbf{ n} } \|_{L^2(F)} \right) \left( h^{-1/2} \| \jump{v} \|_{L^2(F)}  \right) \\
\le  &   \left(  \sum_{j=1,2} \| \sqrt{\beta} \nabla v \|_{L^2(T^j_F)}   \right)
 \left( h^{-1/2} C\frac{\beta^+}{\sqrt{\beta^-}} \| \jump{v} \|_{L^2(F)}   \right) \\
 \le  & \frac{1}{12}\left(  \sum_{j=1,2} \| \sqrt{\beta} \nabla v \|^2_{L^2(T^j_F)}   \right) + 6\left( C^2
\frac{(\beta^+)^2}{\beta^-h} \| \jump{v} \|^2_{L^2(F)}  \right).
\end{split}
\end{equation}
Summing \eqref{thm_coer_eq5} over all the interface faces, we have
\begin{equation}
\label{thm_coer_eq6}
\verti{ (\epsilon - 1) \sum_{F\in \mathcal{F}^i_h} \int_F \aver{ \beta \nabla v\cdot \mathbf{ n} } \jump{v} ds }
\le  \frac{1}{2}\sum_{T\in\mathcal{T}_h} \| \sqrt{\beta} \nabla v \|^2_{L^2(T)}  + 6C^2
\frac{(\beta^+)^2/\beta^-}{h} \sum_{F\in \mathcal{F}^i_h} \| \jump{v} \|^2_{L^2(F)} .
\end{equation}
Similarly, using \eqref{rem_trace_inequa_eq} again, we have
\begin{equation}
\label{thm_coer_eq7}
\sum_{F\in \mathcal{F}^i_h}\| h^{1/2} \aver{\beta\nabla v\cdot\mathbf{ n} } \|^2_{L^2(F)} \le  3\sum_{T\in\mathcal{T}_h} C^2 \frac{(\beta^+)^2}{\beta^-}  \| \sqrt{\beta} \nabla v \|^2_{L^2(T)}.
\end{equation}
Taking $\sigma = 12C^2\frac{(\beta^+)^2}{\beta^-}$ and putting \eqref{thm_coer_eq6} and \eqref{thm_coer_eq7} into \eqref{thm_coer_eq1}, we have
\begin{equation}
\begin{split}
\label{thm_coer_eq8}
a_h(v,v) \ge  & \sum_{T\in\mathcal{T}_h} \left( 1-\frac{1}{2} - \frac{1}{4} \right) \sum_{T\in\mathcal{T}_h}  \| \sqrt{\beta} \nabla v \|^2_{L^2(T)} + \left( \sigma -  6C^2 
\frac{(\beta^+)^2}{\beta^-} \right) h^{-1} \sum_{F\in \mathcal{F}_h^i} \| \jump{v} \|^2_{L^2(F)} \\
 &+ \sum_{F\in \mathcal{F}^i_h } (\sigma)^{-1} \| h^{1/2} \{\beta\nabla v\cdot\mathbf{ n} \} \|^2_{L^2(F)}  \ge  \frac{1}{4} \tbar v \tbar^2.
 \end{split}
\end{equation}
\end{proof}

Let $u^s_E\in H^2_0(\Omega)$ be the Sobolev extension of $u^s=u|_{\Omega^s}$ from $\Omega^s$ to $\Omega$, $s=\pm$. According to the boundedness of Sobolev extensions Theorem 7.25 in \cite{2001GilbargTrudinger} and Poincar\'e inequality, there holds
\begin{equation}
\label{soblev_ext}
|u^s_E|_{H^1(\Omega)} + |u^s_E|_{H^2(\Omega)} \le  C_E( |u^s|_{H^1(\Omega^s)} + |u^s|_{H^2(\Omega^s)} ),~~~~ s=\pm,
\end{equation}
for some constant $C_E$ only depending on $\Omega^{\pm}$. Now we recall the nodal interpolation $\mathcal{I}_h$ for IFE functions from \cite{2020GuoLin}:
\begin{equation}
\label{IFE_interp}
\mathcal{I}_h : H^2(\Omega^+\cup \Omega^-) \rightarrow S_h(\Omega), ~~~ \mathcal{I}_hu(X) = u(X), ~ \forall X\in \mathcal{N}_h.
\end{equation}
According to Theorem 4.3 from \cite{2020GuoLin}, if $u$ satisfies the jump conditions, $\mathcal{I}_hu-u$ has the optimal convergence rate with respect to the mesh size $h$ on each patch $\omega_T$ defined in \eqref{patch} of an interface element $T$, namely,
\begin{equation}
\label{IFE_interp_error}
  | u - \mathcal{I}_hu |_{H^j(\omega_T)} \le  C \frac{\beta^+}{\beta^-} h^{2-k} \sum_{j=1,2}( |u^+_E|_{H^j(\omega_T)} +  |u^-_E|_{H^j(\omega_T)} ), ~ k=0,1,2.
\end{equation}
We can further use this result to estimate the interpolation errors in terms of the energy norm.
\begin{lemma}
\label{lem_interp_enrg_error}
Assume that the mesh $\mathcal{T}_h$ is fine enough, then we have
\begin{equation}
\label{lem_interp_enrg_error_eq0}
\tbar u - \mathcal{I}_hu \tbar \le  C \frac{(\beta^+)^2}{(\beta^-)^{3/2}} h \sum_{j=1,2}( |u^+|_{H^j(\Omega^+)} +  |u^-|_{H^j(\Omega^-)} ).
\end{equation}
\end{lemma}
\begin{proof}
First of all, \eqref{IFE_interp_error} and the standard estimate of the Lagrange interpolation for finite element functions give
\begin{equation}
\label{lem_interp_enrg_error_eq1}
 \| \sqrt{\beta} \nabla (u - \mathcal{I}_hu)\|_{L^2(T)}  \le  C \frac{(\beta^+)^{3/2}}{\beta^-} h \sum_{s=\pm}( |u^s_E|_{H^1(\omega_T)} +  |u^s_E|_{H^2(\omega_T)} ).
\end{equation}
For the second term in \eqref{energy_norm}, for each $F\in\mathcal{F}^i_h$, we denote $T^1_F$ and $T^2_F$ as the two neighbor elements. Using the trace inequality and \eqref{IFE_interp_error}, we have
\begin{equation}
\begin{split}
\label{lem_interp_enrg_error_eq2}
\sqrt{\sigma} \| h^{-1/2} \jump{u - \mathcal{I}_hu} \|_{L^2(F)} & \le  C \frac{\beta^+}{(\beta^-)^{1/2}} \sum_{r=1,2} ( h^{-1}\| u - \mathcal{I}_hu \|_{L^2(T^r_F)} + | u - \mathcal{I}_hu |_{H^1(T^r_F)} ) \\
& \le  C \frac{(\beta^+)^2}{(\beta^-)^{3/2}} h \sum_{r=1,2} \sum_{s=\pm} ( | u^s_E  |_{H^1(\omega_{T^r_F})} + | u^s_E  |_{H^2(\omega_{T^r_F})}).
\end{split}
\end{equation}
For the third term in \eqref{energy_norm}, by similar derivation we have
\begin{equation}
\label{lem_interp_enrg_error_eq3}
\frac{1}{\sqrt{\sigma}} \| h^{1/2} \{\beta\nabla (u - \mathcal{I}_hu)\cdot\mathbf{ n} \} \|_{L^2(F)} \le  C \frac{\beta^+}{(\beta^-)^{1/2}} h \sum_{r=1,2} \sum_{s=\pm} ( | u^s_E  |_{H^1(\omega_{T^r_F})} + | u^s_E  |_{H^2(\omega_{T^r_F})} ).
\end{equation}
Summing the estimates above over all the elements and interface faces, using the finite overlapping of the patches $\omega_T$, $T\in\mathcal{T}^i_h$, and applying the boundedness \eqref{soblev_ext}, we have the desired result. 
\end{proof}

Now we are ready to present the error estimates of IFE solutions in terms of the energy norm and $L^2$ norm. We note that the key difficulty for these two estimates is the treatment of the non-consistence of the IFE scheme.

\begin{theorem}
\label{thm_error_bound}
Assume that the mesh $\mathcal{T}_h$ is fine enough, and assume that $\sigma$ is large enough such that Theorem \ref{thm_coer} holds, then there holds
\begin{equation}
\label{thm_error_bound_eq0}
\tbar u - u_h \tbar \le  Ch \frac{(\beta^+)^2}{(\beta^-)^{5/2}} \sum_{j=1,2}\left( | \beta^- u^-|_{H^j(\Omega^-)} + |\beta^+u^+|_{H^j(\Omega^+)} \right).
\end{equation}
\end{theorem}
\begin{proof}
We note that the exact solution $u$ may not exactly satisfy the PPIFE scheme \eqref{ppife} since the IFE functions on interface elements may not be continuous across the interface. In fact, testing \eqref{model} with any $v\in V_h(\Omega)$ and using integration by parts, we can write
\begin{equation}
\begin{split}
\label{thm_error_bound_eq1}
 &\sum_{T\in\mathcal{T}_h} \int_T \beta \nabla u \cdot \nabla v dX 
 - \sum_{F\in \mathcal{F}^i_h} \int_F \aver{ \beta \nabla u\cdot \mathbf{ n} } \jump{v} ds  
 + \epsilon \sum_{F\in \mathcal{F}^i_h} \int_F \aver{ \beta \nabla v\cdot \mathbf{ n} } \jump{u} ds \\
 &+ \sum_{F\in \mathcal{F}^i_h} \frac{\sigma }{h} \int_F \jump{u}\jump{v} ds 
 -  \sum_{T\in\mathcal{T}^i_h} \int_{T\cap \Gamma}  \beta \nabla u\cdot \mathbf{ n} \jump{v} ds = \int_{\Omega} fv dX.
 \end{split}
\end{equation}
Due to the flux jump condition of $u$, we focus on $\beta \nabla u\cdot \mathbf{ n}=\beta \nabla u^-\cdot \mathbf{ n}$ on $\Gamma$. We define the following bilinear form
\begin{equation}
\label{thm_error_bound_eq2}
b_h(w,v) = \sum_{T\in\mathcal{T}^i_h} \int_{T\cap \Gamma}  \beta \nabla w^-\cdot \mathbf{ n} \jump{v} ds, ~~~ \forall v,w \in V_h(\Omega).
\end{equation}
Combining \eqref{thm_error_bound_eq1} and the IFE scheme, we obtain
\begin{equation}
\label{thm_error_bound_eq3}
a_h(u,v) - b_h(u,v) = a_h(u_h,v).
\end{equation}
We need to estimate the bound of $b_h(u,v)$ for each $v\in S_h(\Omega)$. Using Theorem \ref{thm_trace_interface} and the trace inequality (Lemma 3.2 in \cite{2016WangXiaoXu}), we have
\begin{equation}
\begin{split}
\label{thm_error_bound_eq4}
|b_h(u,v)| &\le  \sum_{T\in\mathcal{T}^i_h} \| \beta^- \nabla u^-\cdot \mathbf{ n} \|_{L^2(\Gamma\cap T)} \| [v] \|_{L^2(\Gamma\cap T)} \\
& \le  C \frac{\sqrt{\beta^+}}{\beta^-} \sum_{T\in\mathcal{T}^i_h} ( h^{-1/2}|\beta^-u^-_E|_{H^1(T)} +  h^{1/2}|\beta^-u^-_E|_{H^2(T)} ) h^{3/2} \| \sqrt{\beta} \nabla v \|_{L^2(T)} \\
& \le  C \frac{\sqrt{\beta^+}}{\beta^-} h ( | \beta^- u^-_E |_{H^1(\Omega)}  + | \beta^- u^-_E |_{H^2(\Omega)}  ) \tbar  v \tbar.
\end{split}
\end{equation}
Now we consider the Lagrange interpolation operator $\mathcal{I}_h$ and write
\begin{equation}
\label{thm_error_bound_eq5}
a_h(u_h - \mathcal{I}_h u,v) = a_h(u- \mathcal{I}_hu,v) - b_h(u,v).
\end{equation}
Taking $v_h = u_h - \mathcal{I}_h u\in S_h(\Omega)$ and applying coercivity in Theorem \ref{thm_coer}, the boundedness in Theorem \ref{thm_bound} as well as \eqref{thm_error_bound_eq4}, we arrive at
\begin{equation}
\label{thm_error_bound_eq6}
\tbar u_h - \mathcal{I}_h u \tbar^2 \le  C \tbar u_h - \mathcal{I}_h u \tbar\tbar u - \mathcal{I}_h u \tbar + C \frac{\sqrt{\beta^+}}{\beta^-} h ( | \beta^-u^-_E |_{H^1(\Omega)}  + | \beta^-u^-_E |_{H^2(\Omega)}  ) \tbar  u_h - \mathcal{I}_h u \tbar.
\end{equation}
By the optimal approximation of $\mathcal{I}_hu$ given in Lemma \ref{lem_interp_enrg_error}, \eqref{thm_error_bound_eq6} yields
\begin{equation}
\label{thm_error_bound_eq7}
\tbar u_h - \mathcal{I}_h u\tbar \le  C \frac{(\beta^+)^3}{(\beta^-)^{3/2}} h \sum_{j=1,2}( | u^-_E |_{H^j(\Omega)}  + | u^+_E |_{H^j(\Omega)} ) + C \frac{\sqrt{\beta^+}}{\beta^-} h ( | \beta^-u^-_E |_{H^1(\Omega)}  + | \beta^-u^-_E |_{H^2(\Omega)}  ).
\end{equation}
Clearly, the triangle inequality together with \eqref{thm_error_bound_eq7} and \eqref{soblev_ext} yields the desired result.
\end{proof}

\begin{remark}
\label{rem_regularity}
The regularity of elliptic interface problems \cite{2010ChuGrahamHou,2002HuangZou} gives that
\begin{equation}
\label{rem_regularity_eq1}
\sum_{j=1,2}\left( | \beta^- u^-|_{H^j(\Omega^-)} + |\beta^+u^+|_{H^j(\Omega^+)} \right) \le  C_{reg} \| f \|_{L^2(\Omega)}.
\end{equation}
where this constant $C_{reg}$ only depends on the $\Omega^{\pm}$. So Theorem \ref{thm_error_bound} actually yields
\begin{equation}
\label{rem_regularity_eq2}
\tbar u - u_h \tbar \le  Ch \frac{(\beta^+)^2}{(\beta^-)^{5/2}} \| f \|_{L^2(\Omega)}.
\end{equation}
\end{remark}

\begin{theorem}
\label{thm_l2_error_bound}
Under the conditions of Theorem \ref{thm_error_bound}, there holds
\begin{equation}
\label{thm_l2_error_bound_eq0}
\| u - u_h \|_{L^2(\Omega)} \le  Ch^2 \frac{(\beta^+)^4}{(\beta^-)^{5}} \| f \|_{L^2(\Omega)}.
\end{equation}
\end{theorem}
\begin{proof}
We use the duality argument. Define an auxiliary function $z\in H^2(\Omega)$ to the interface problem \eqref{model} with the right hand side $f$ replaced by $u-u_h\in L^2(\Omega)$. Again, we consider the Lagrange interpolation operator $\mathcal{I}_h$ for IFE functions \eqref{IFE_interp}. Testing this auxiliary equation with $u - u_h$ and using the similar derivation as \eqref{thm_error_bound_eq1}, we have
\begin{equation}
\begin{split}
\label{thm_l2_error_bound_eq1}
\| u - u_h \|_{L^2(\Omega)}^2 & = a_h(z, u -u_h ) - b_h(z,u - u_h) \\
& = a_h(z - \mathcal{I}_h z, u -u_h ) + b_h(u, \mathcal{I}_hz) - a_h(z-z_h,u-u_h)
\end{split}
\end{equation}
where in the second equality we have used the identity \eqref{thm_error_bound_eq3} again. Lemma \ref{lem_interp_enrg_error} for $z$ and \eqref{rem_regularity_eq2} show
\begin{equation}
\begin{split}
\label{thm_l2_error_bound_eq2}
& a_h(z - \mathcal{I}_h z, u -u_h )  \le  C\tbar z-  \mathcal{I}_hz \tbar \tbar  u -u_h \tbar \\
 \le & C \frac{(\beta^+)^4}{(\beta^-)^{5}} h^2 \sum_{j=1,2}( |\beta^+ z^+|_{H^j(\Omega^+)} +  |\beta^-z^-|_{H^j(\Omega^-)} )  \| f \|_{L^2(\Omega)} 
 \le  C \frac{(\beta^+)^4}{(\beta^-)^{5}} h^2 \| u - u_h \|_{L^2(\Omega)} \| f \|_{L^2(\Omega)} .
\end{split}
\end{equation}
Similarly, for the third term on the right side of \eqref{thm_l2_error_bound_eq1}, we have
\begin{equation}
\label{thm_l2_error_bound_eq3}
a_h(z-z_h,u-u_h) \le  C \frac{(\beta^+)^4}{(\beta^-)^{5}} h^2 \| u - u_h \|_{L^2(\Omega)} \| f \|_{L^2(\Omega)}.
\end{equation}
For the second term on the right side of \eqref{thm_l2_error_bound_eq1}, using Theorem \ref{thm_interface_interface}, we obtain
\begin{equation}
\begin{split}
\label{thm_l2_error_bound_eq4}
 |b_h(u,\mathcal{I}_hz)|&  \le  Ch^2 \sum_{T\in\mathcal{T}^i_h} \| \beta^- \nabla u^-\cdot \mathbf{ n} \|_{L^2(\Gamma\cap T)} | \jump{\mathcal{I}_hz} |_{H^1(\Gamma\cap T)} \le  Ch^2 \| \beta^- \nabla u^-\cdot \mathbf{ n} \|_{L^2(\Gamma)} |  \jump{\mathcal{I}_hz} |_{H^1(\Gamma)}  \\
&
\le   Ch^2 \| \beta^- \nabla u^-\cdot \mathbf{ n} \|_{L^2(\Gamma)}\big( |  \jump{\mathcal{I}_hz-z} |_{H^1(\Gamma)} +  |  \jump{z} |_{H^1(\Gamma)}\big).
\end{split}
\end{equation}
Let $z^{\pm}_E$ be the Sobolev extensions of $z^{\pm}=z|_{\Omega^{\pm}}$ from $\Omega^{\pm}$ to $\Omega$. By the trace inequality Lemma 3.2 in \cite{2016WangXiaoXu} and \eqref{IFE_interp_error} on each interface element, we have
\begin{equation}
\begin{split}
\label{thm_l2_error_bound_eq5}
 |  \jump{\mathcal{I}_hz-z} |_{H^1(\Gamma)}  & \le   \sum_{T\in\mathcal{T}^i_h} \sum_{s=\pm} C (h^{-1/2}| \mathcal{I}_hz^s_E-z^s_E |_{H^1(T)} + h^{1/2}| \mathcal{I}_hz^s_E-z^s_E |_{H^2(T)} ) \\
&  \le C \frac{\beta^+}{(\beta^-)^2} h^{1/2}  \sum_{j=1,2}( |\beta^+z^+|_{H^j(\Omega^-)} +  |\beta^-z^-|_{H^j(\Omega^+)} )  
\le  C \frac{\beta^+}{(\beta^-)^2} h^{1/2} \| u - u_h \|_{L^2(\Omega)}
\end{split}
\end{equation}
where in the second inequality we have used the boundedness for Sobolev extensions. In addition, by the trace inequality from $\Gamma$ to $\Omega^{\pm}$, we have
\begin{equation}
\label{thm_l2_error_bound_eq6}
| \jump{z} |_{H^1(\Gamma)} \le  C \frac{1}{\beta^-} \sum_{j=1,2}( |\beta^+z^+|_{H^j(\Omega^-)} +  |\beta^-z^-|_{H^j(\Omega^+)} ) 
\le  C \frac{1}{\beta^-} \| u - u_h \|_{L^2(\Omega)}.
\end{equation}
Putting \eqref{thm_l2_error_bound_eq5} and \eqref{thm_l2_error_bound_eq6} into \eqref{thm_l2_error_bound_eq4} and applying the trace inequality to $\nabla u^-\cdot\bfn$ from $\Gamma$ to $\Omega^-$, we get
\begin{equation}
\label{thm_l2_error_bound_eq7}
|b_h(u,\mathcal{I}_hz)| \le  C \frac{\beta^+}{(\beta^-)^2} h^2 ( | \beta^- u^-_E |_{H^1(\Omega)}  + | \beta^- u^-_E |_{H^2(\Omega)}  ) \|  u - u_h \|_{L^2(\Omega)}.
\end{equation}
Substituting \eqref{thm_l2_error_bound_eq2}, \eqref{thm_l2_error_bound_eq3} and \eqref{thm_l2_error_bound_eq7} into \eqref{thm_l2_error_bound_eq1}, we have \eqref{thm_l2_error_bound_eq0}.
\end{proof}

\begin{remark}
We are able to specify how the error bound depends on the material property parameters $\beta^{\pm}$ at each step throughout the analysis. But it is important to note that the dependence on $\beta^{\pm}$ in the final estimates in Theorem \ref{thm_error_bound}, \eqref{rem_regularity_eq2} and Theorem \ref{thm_l2_error_bound} is due to the limitation of our analysis approach, since we do not observe such a severe effect from $\beta^{\pm}$ in computation. How to achieve the optimal error bound with respect to $\beta^{\pm}$ is an interesting topic in our future research.
\end{remark}


\section{Numerical Experiments}

In this section, we report some numerical experiments to demonstrate the performance of our IFE method. In the first three examples, we present artificial interface problems where we know the analytical function of the interface surface and the exact solution. In Example 4, we present a real-world interface model of which the interface has a dabbling-duck shape but only the cloud-point data are available.

\subsubsection*{Example 1 (Plane Interface: Recovering Exact Solutions)}
In the first example, we compare the performance of the PPIFE and the classic IFE methods when the exact solutions are contained in the IFE spaces. Let $\Omega = (-1,1)^3$ and consider a planar interface $\Gamma = \{(x,y,z)\in\Omega: \gamma(x,y,z) = 0\}$ where 
\[\gamma(x,y,z) = \frac{1}{\sqrt{2}}(x + z-\pi/10).\]
Let the exact solution be 
\begin{equation}
u(x,y,z) = 
\left\{
\begin{split}
&\frac{1}{\beta^-}\gamma(x,y,z)~~~& \text{in}~\Omega^- := \{(x,y,z)\in\Omega:\gamma(x,y,z)<0\},\\
&\frac{1}{\beta^+}\gamma(x,y,z)~~~& \text{in}~\Omega^+ := \{(x,y,z)\in\Omega:\gamma(x,y,z)>0\}.
\end{split}
\right.
\end{equation}
Our computation is carried out on a family of uniform Cartesian meshes consisting of $N^3$ cuboids. We report errors in the discrete $L^\infty$, $L^2$, and $H^1$-norms, denoted by $e^\infty_h$, $e^0_h$, and $e^1_h$, respectively. The errors for both PPIFE and IFE methods are reported in Table \ref{table: linear error table}. We note that the PPIFE method actually recovers the exact solutions with no approximation errors (only round-off errors are observed). This suggests that PPIFE is a consistent numerical algorithm; namely, if the exact solution is a piecewise linear function separated by a planar interface, then the PPIFE method will reproduce the exact solution. In contrast, the classical IFE method without penalty cannot generate exact solution due to the inconsistency caused by the discontinuities of IFE functions across interface faces. 
\begin{table}[h!]
\begin{small}
\centering
\begin{tabular}{|c| c c c |c c c| }
\hline
&\multicolumn{3}{c|}{PPIFE} &\multicolumn{3}{c|}{IFE}\\
\hline
N     &$e^\infty_h$   & $e^0_h$  & $e^{1}_h$ & $e^\infty_h$   & $e^0_h$  & $e^{1}_h$ \\
\hline 
10    &5.50E-15 & 2.29E-15 & 1.74E-14 & 7.23E-3 & 3.07E-3 & 4.95E-2 \\ 
20    &1.92E-15 & 6.37E-16 & 8.93E-15 & 2.75E-3 & 1.03E-3 & 2.47E-2 \\ 
30    &5.33E-15 & 3.87E-15 & 1.75E-14 & 1.87E-3 & 4.93E-4 & 2.24E-2\\
40    &3.11E-15 & 2.04E-15 & 1.73E-14 & 1.81E-3 & 5.65E-4 & 2.37E-2\\ 
\hline
\end{tabular}
\caption{PPIFE errors and the convergence rates for linear solution}
\label{table: linear error table}
\end{small}
\end{table}

\subsubsection*{Example 2 (Sphere Interface)}
In the second example, we let $\Omega = (-1,1)^3$ and let the interface be a sphere $\Gamma =\{(x,y,z): \gamma(x,y,z)=0\}$ where $\gamma(x,y,z) =x^2+y^2+z^2 -r^2$. The exact solution is given by
\begin{equation}
u(x,y,z) = 
\left\{
\begin{split}
&-\cos\left(\frac{\pi(x^2+y^2+z^2)}{2r^2}\right)~~~& \text{in}~\Omega^- := \{(x,y,z)\in\Omega:\gamma(x,y,z)<0\},\\
&x^2+y^2+z^2-r^2~~~& \text{in}~\Omega^+ := \{(x,y,z)\in\Omega:\gamma(x,y,z)>0\}.
\end{split}
\right.
\end{equation}
The parameters are chosen to be $r=\pi/4$ and $\beta^- = 1$, and $\beta^+ = \frac{\pi}{2r^2}\approx 2.5465$. Our computation is carried out on a family of uniform Cartesian meshes consisting of $N^3$ cuboids. We start from a coarse mesh with $N=20$ and stretch to a very fine mesh with $N=160$ by an increment of $10$ more partitions in each direction for each finer mesh. We report errors in the discrete $L^\infty$, $L^2$, and $H^1$-norms for both PPIFE and classical IFE methods. See Figure \ref{fig: convergence ex2} for a comparison of the performances in all three norms. Using linear regression, the errors of PPIFE solution obey
\[e_h^\infty\approx 9.898h^{2.034},~~~e_h^0\approx 9.093h^{2.004},~~~e_h^1\approx 9.017h^{0.999}.\]

Our numerical results show that the PPIFE method converge optimally in both $L^2$ and $H^1$ norms, which confirms our theoretical error bounds \eqref{thm_error_bound_eq0} and \eqref{thm_l2_error_bound_eq0}. We also observe that the convergence rate in $L^\infty$ norm is also optimal for the PPIFE method, although we don't know how to theoretically prove it yet. The PPIFE method clearly outperforms the classical IFE method since their errors in 
$L^\infty$ norm are much smaller than IFE method, and their convergence rates in $L^2$ or $H^1$ norm do not deteriorate as the mesh size becomes small. This is consistent with the observation for the 2D case \cite{2015LinLinZhang}. For a more visible comparison, we plot the errors of PPIFE and IFE methods on the interface surface itself in Figure \ref{fig: interface surface error}. Moreover, we plot the errors on five slices of the domain with $y = -0.7$, $-0.35$, $0$, $0.35$, and $0.7$ in Figure \ref{fig: slice error sphere}. From both of these figures, we can clearly see that the PPIFE errors around interface are significantly smaller and are comparable to the errors away from the interface. In contrary, the classical IFE solutions have much larger errors around the interface.

\begin{figure}[h!]
\centering
\includegraphics[width=.32\textwidth]{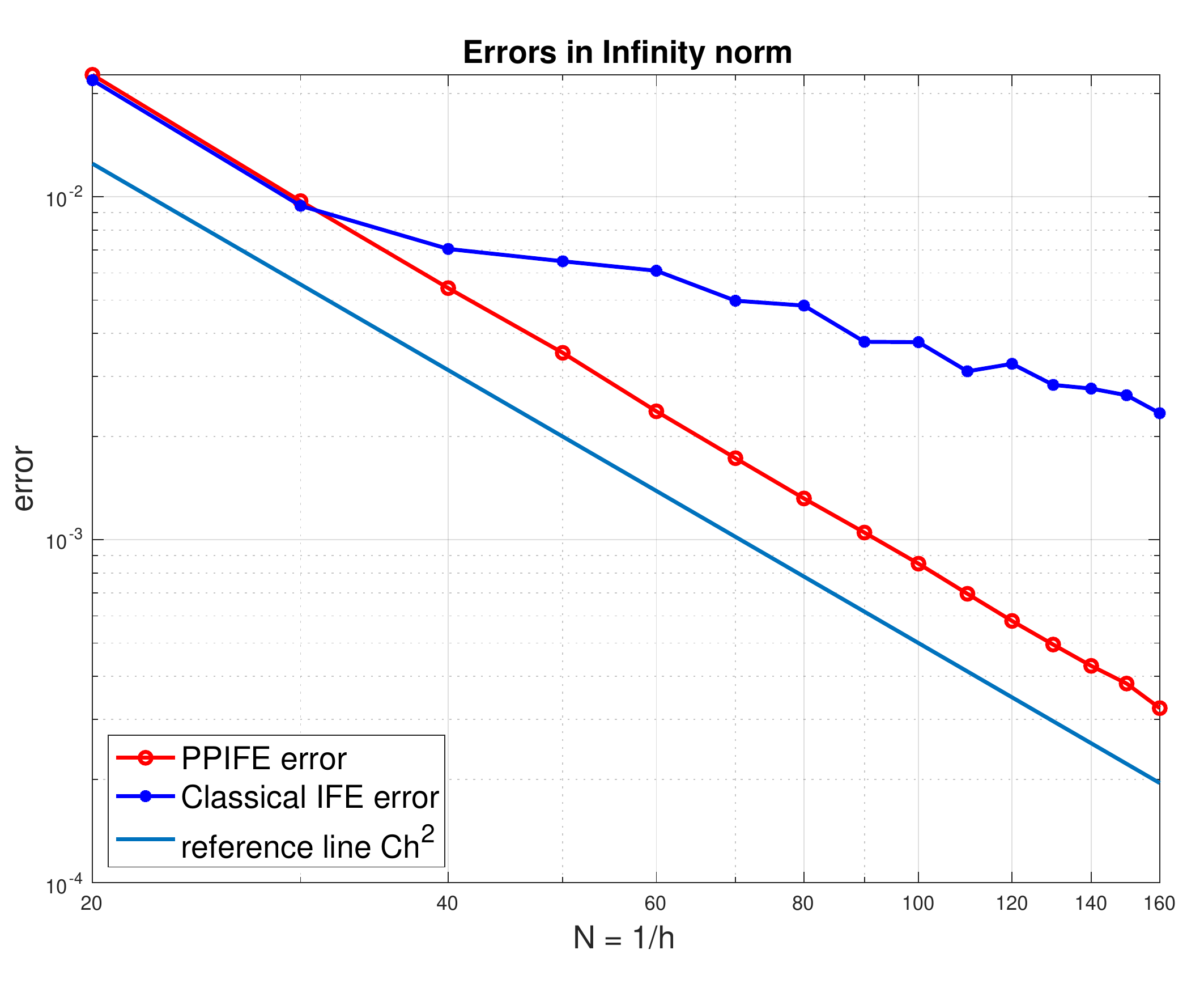}
\includegraphics[width=.32\textwidth]{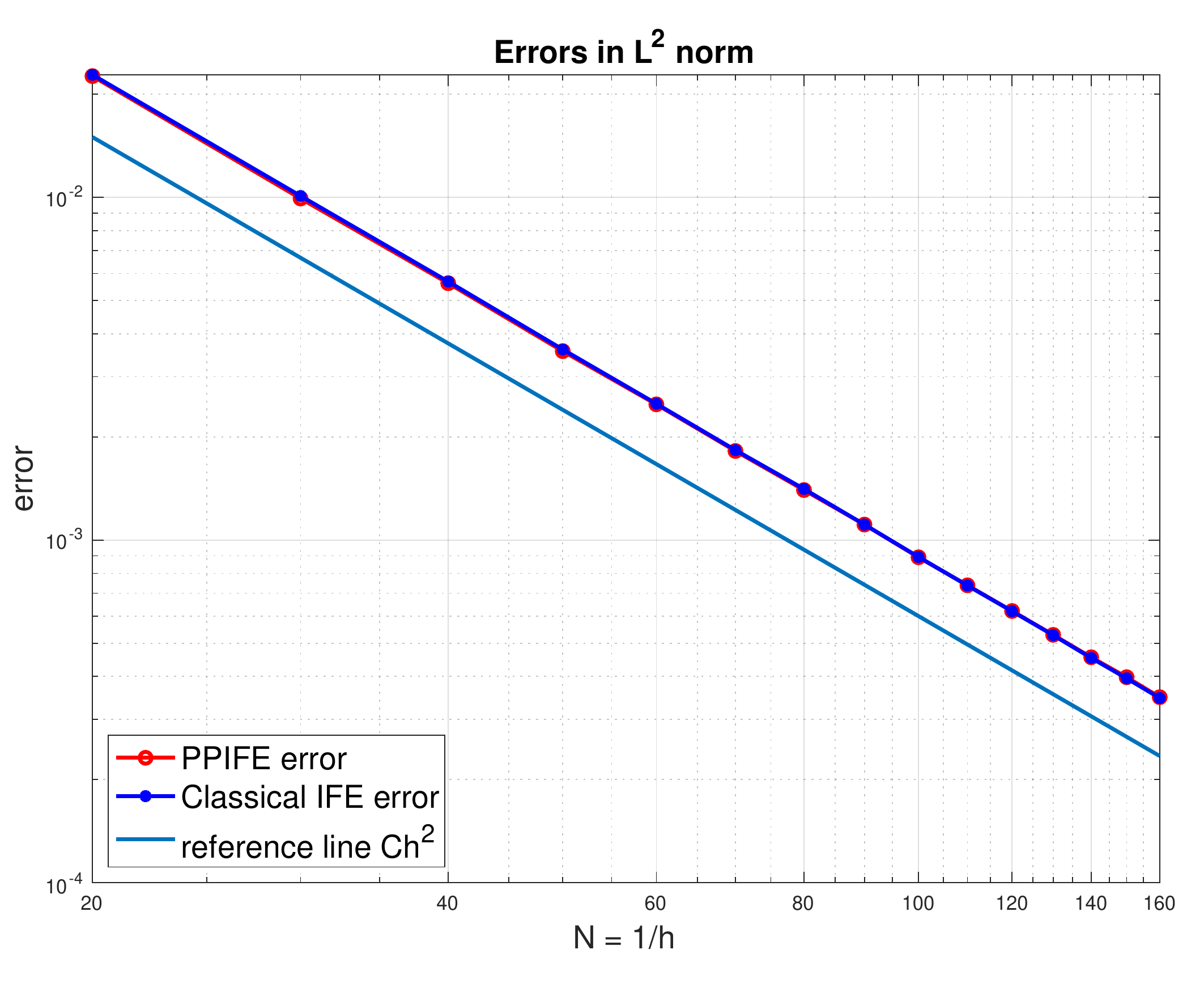}
\includegraphics[width=.32\textwidth]{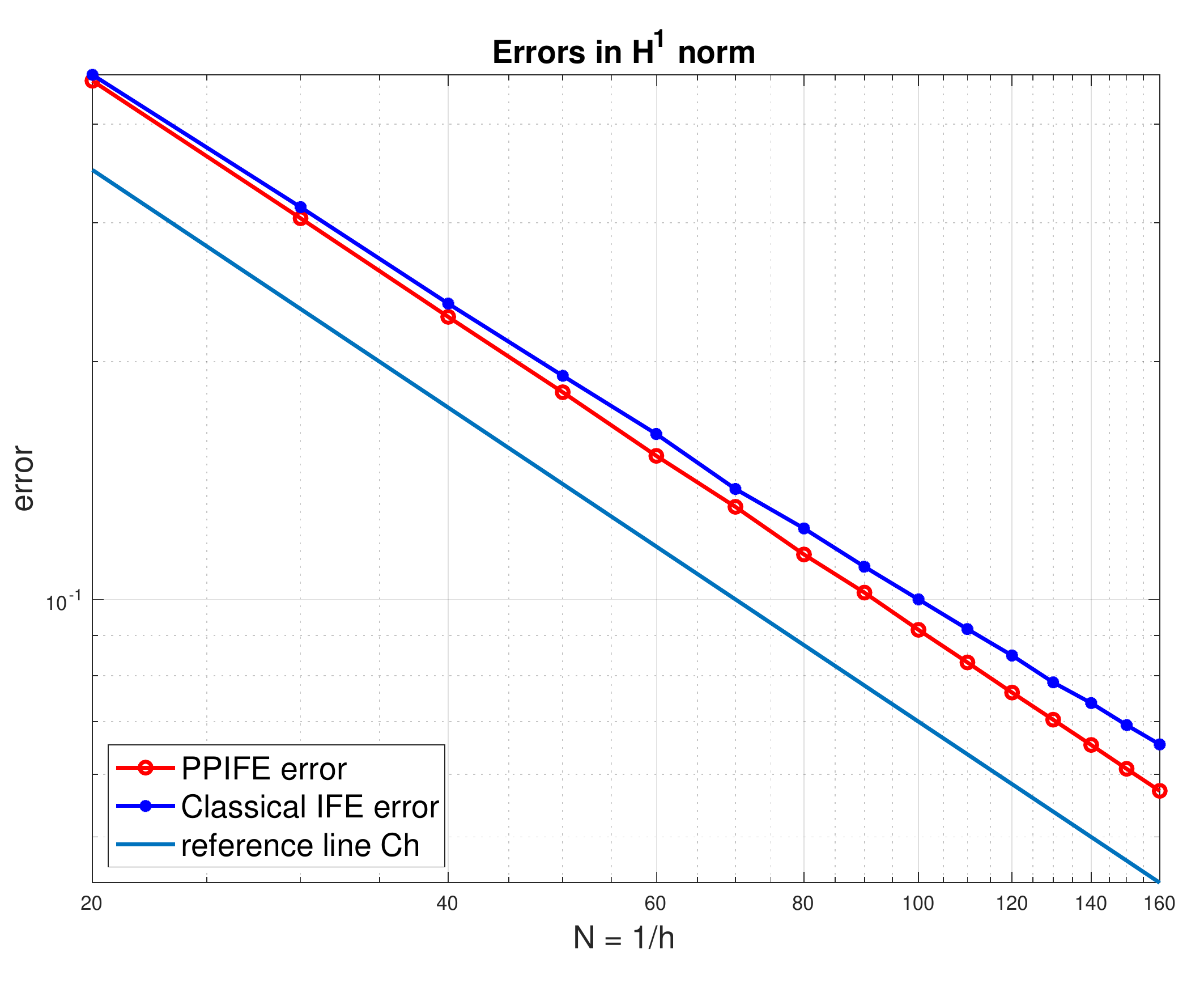}
\caption{Convergence in $L^\infty$, $L^2$, and $H^1$ norms of PPIFE and IFE solutions for Example 2.}
\label{fig: convergence ex2}
\end{figure}

\begin{figure}[h!]
\centering
\includegraphics[width=.3\textwidth]{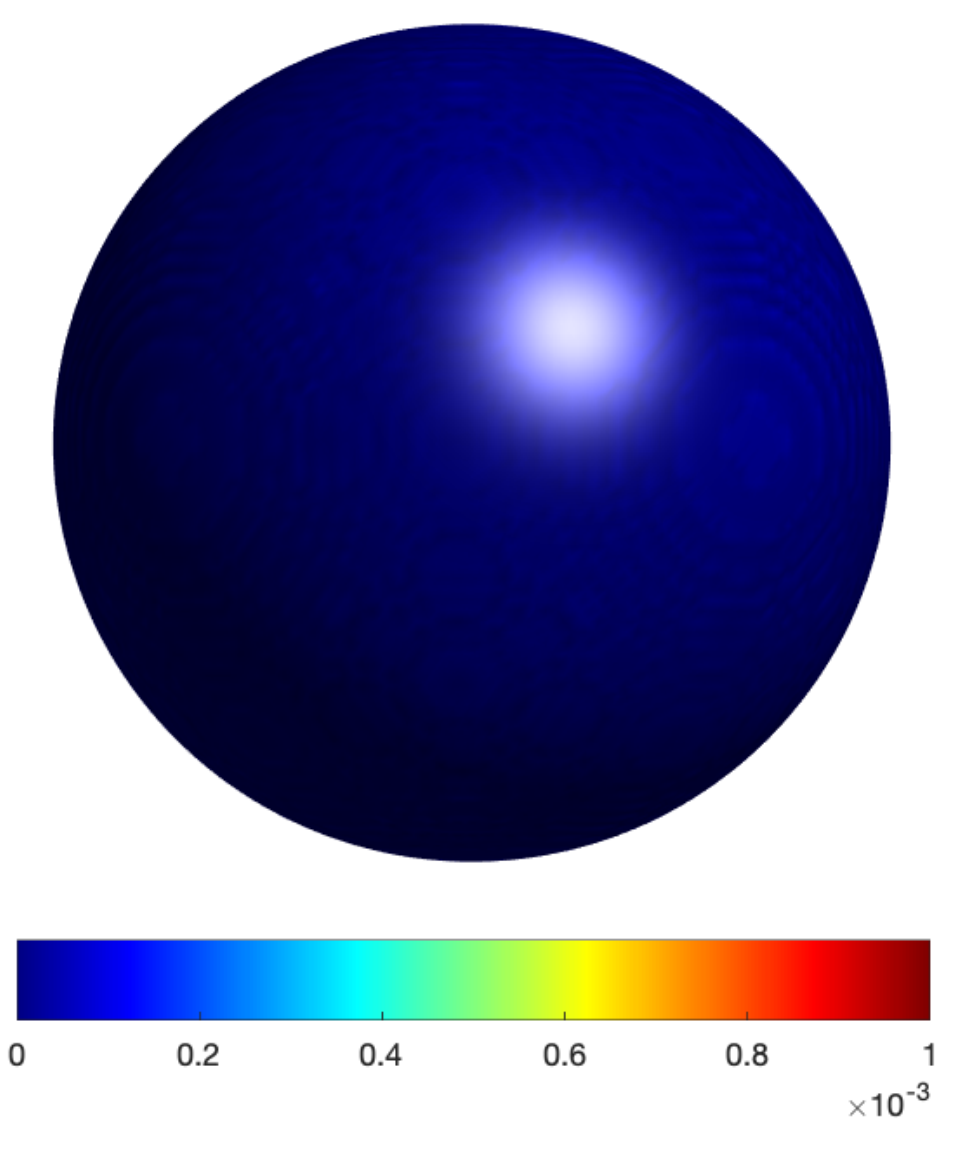}~~~~
\includegraphics[width=.3\textwidth]{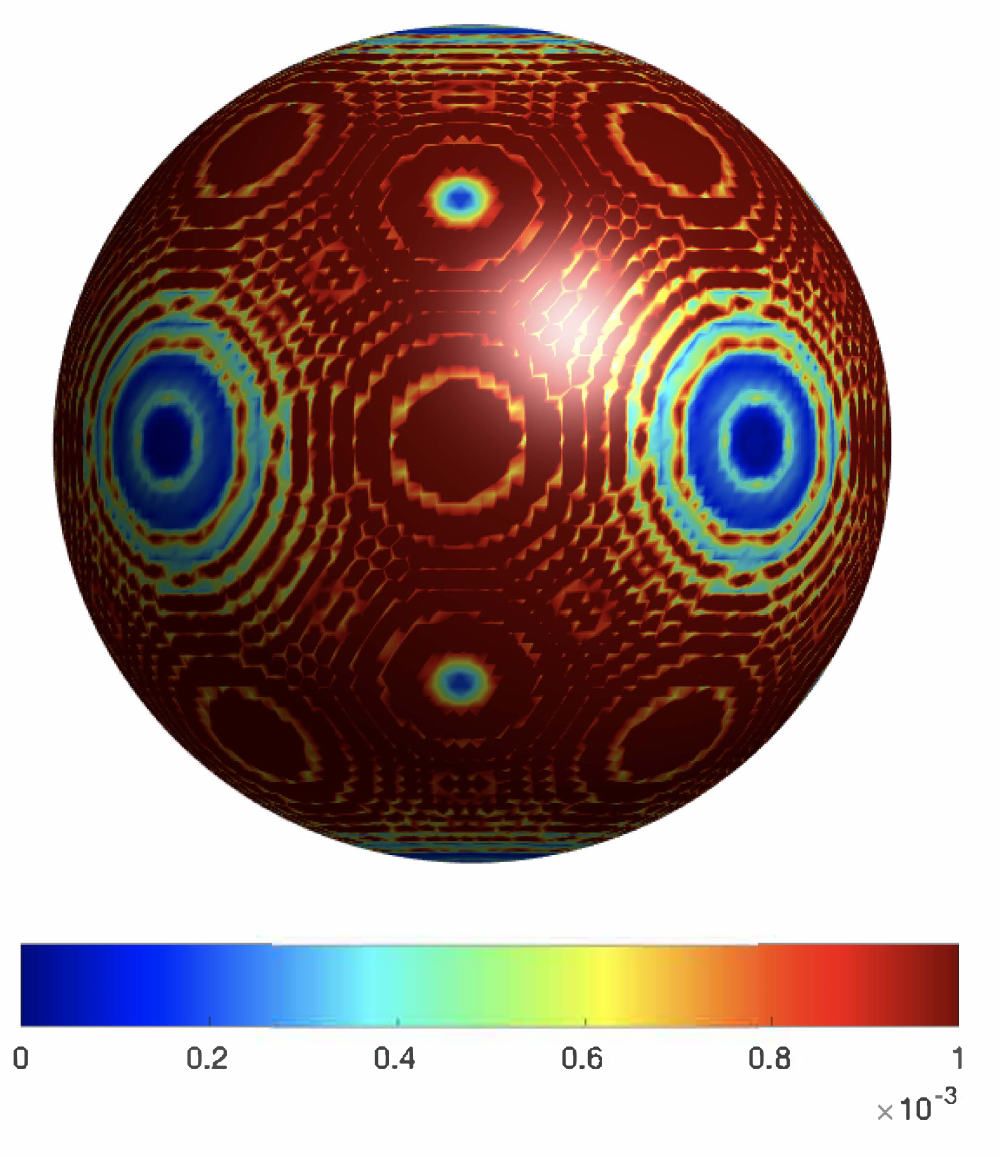}
\caption{A comparison of the PPIFE (left) and the classical IFE (right) errors on the interface surfaces for Example 2, (mesh size $N=100$).}
\label{fig: interface surface error}
\end{figure}

\begin{figure}[h]
\centering
\includegraphics[width=.49\textwidth]{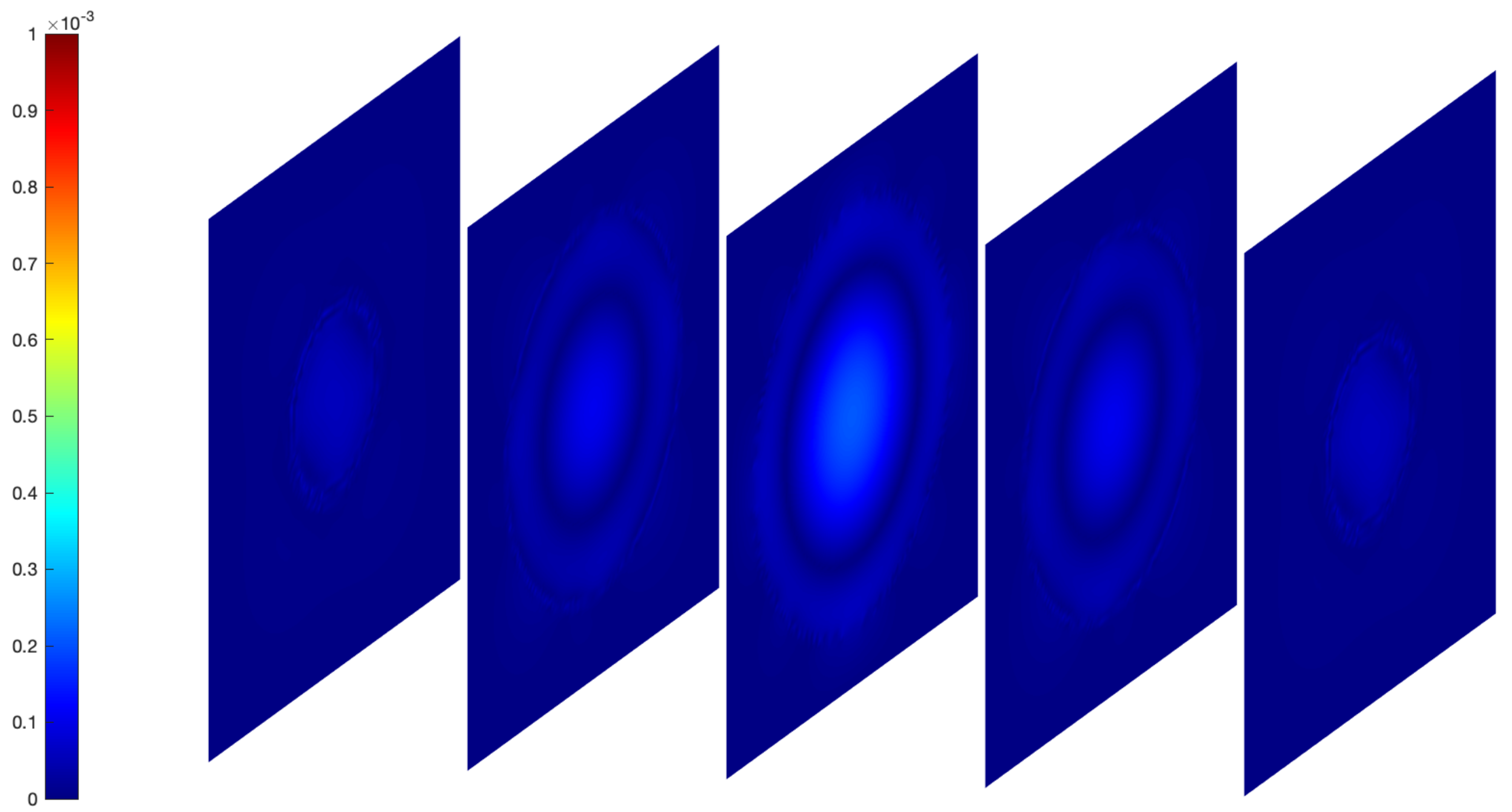}~~
\includegraphics[width=.49\textwidth]{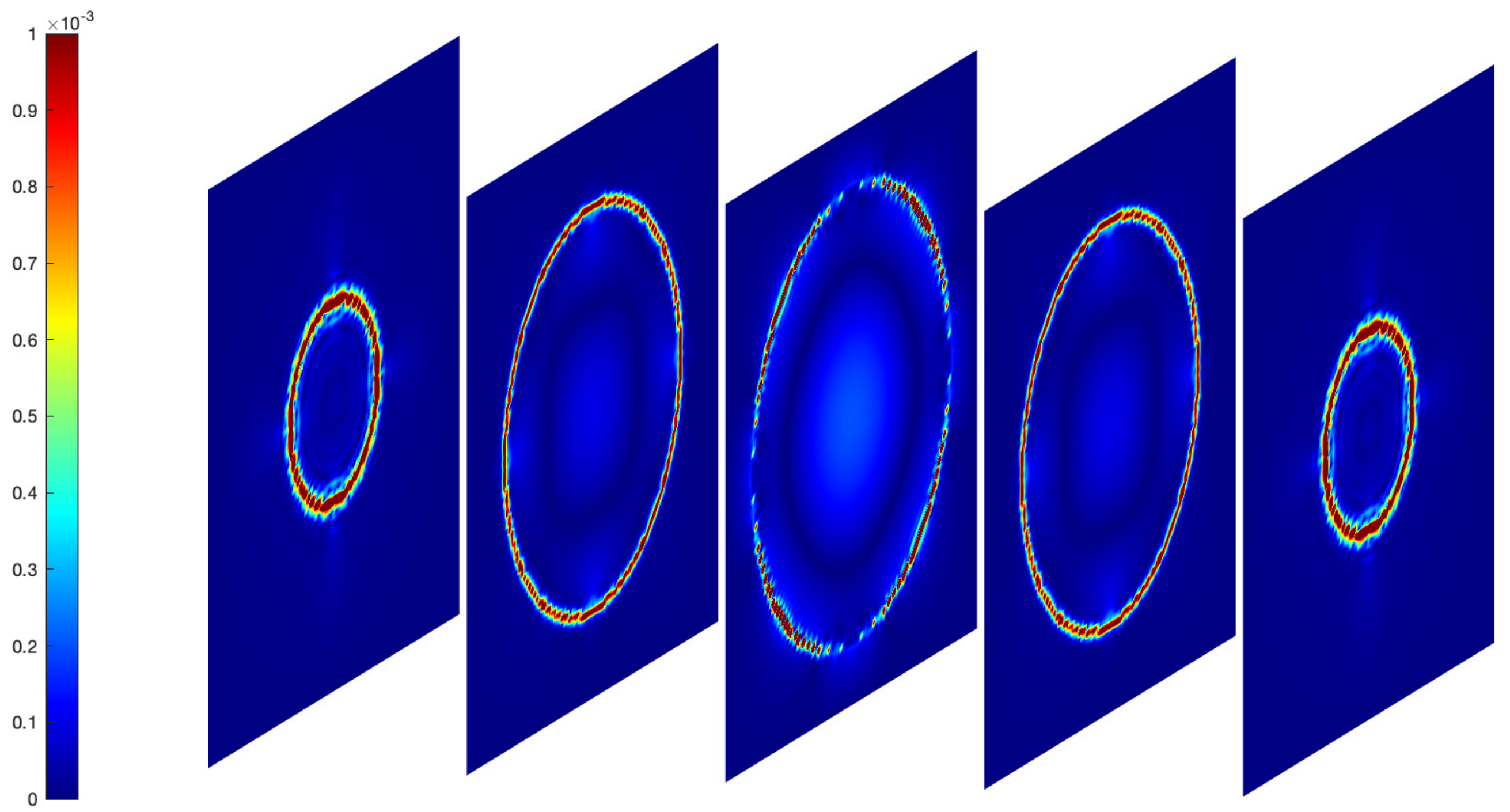}
\caption{A comparison of the PPIFE (left) and the classic IFE (right) errors on slices of the domain at $y = -0.7$, $-0.35$, $0$, $0.35$, and $0.7$ for Example 2, (mesh size $N=100$).}
\label{fig: slice error sphere}
\end{figure}

\subsubsection*{Example 3 (More Completed Topology: An orthocircle Interface)}
In this example, we consider an interface problem with more complicated topology. We let
$\Omega = (-1.2,1.2)^3$, and let the interface be $\Gamma = \{(x,y,z)\in\Omega: \gamma(x,y,z) = 0\}$ where
\[\gamma(x,y,z) = [(x^2+y^2-1)^2+z^2][(x^2+z^2-1)^2+y^2][(y^2+z^2-1)^2+x^2]-0.075^2[1+3(x^2+y^2+z^2)].\]
The shape of the interface is plotted in the left plot of Figure \ref{fig: orthocircle}. This interface problem was reported in \cite{2017ChenWeiWen}. 
Let the exact solution be 
\begin{equation}
u(x,y,z) = 
\left\{
\begin{split}
&\frac{1}{\beta^-}\gamma(x,y,z)~~~& \text{in}~\Omega^- := \{(x,y,z)\in\Omega:\gamma(x,y,z)<0\},\\
&\frac{1}{\beta^+}\gamma(x,y,z)~~~& \text{in}~\Omega^+ := \{(x,y,z)\in\Omega:\gamma(x,y,z)>0\}.
\end{split}
\right.
\end{equation}
The coefficients are chosen to have a larger contrast as $\beta^-=1$ and $\beta^+=100$. The errors of the PPIFE method in all three norms are reported in Figure \ref{fig: convergence ex3}. Again, we can see that overall convergence rates in $L^2$ and $H^1$ norms are close to optimal, which confirms our theoretical results. Using linear regression, the errors obey
\[e_h^\infty\approx 1.365h^{1.340},~~~e_h^0\approx 5.848h^{1.877},~~~e_h^1\approx 7.276h^{1.101}.\]
For comparison, we also report the solutions without imposing the maximum angle condition in Section \ref{sec:model}. See the blue curves in Figure \ref{fig: convergence ex3}. Although we can still see the convergence in all three norms, the magnitudes of errors are larger than the those enforced by the maximum angle condition, see the red curves in Figure \ref{fig: convergence ex3}. We also compare the error surfaces of these two solutions on the interface. See the middle and right plots in Figure \ref{fig: orthocircle}. It can be observed that errors are larger on the interface when the maximal angle condition is not satisfied. These large errors are indicated by the red spots in the right plot in Figure \ref{fig: orthocircle} which are also where the maximal conditions are violated.

\begin{figure}[h!]
\centering
\includegraphics[width=.27\textwidth]{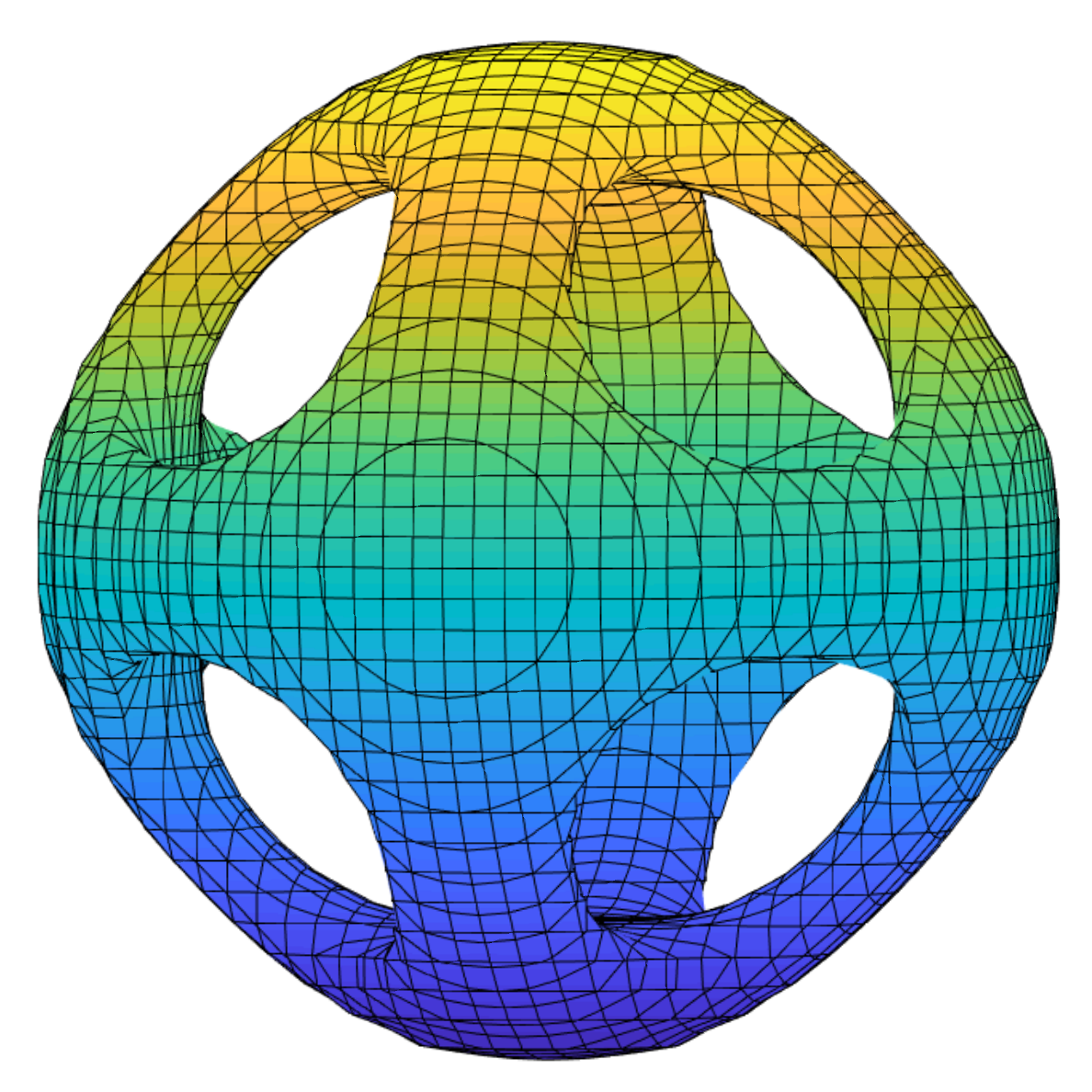}~~~~~~
\includegraphics[width=.32\textwidth]{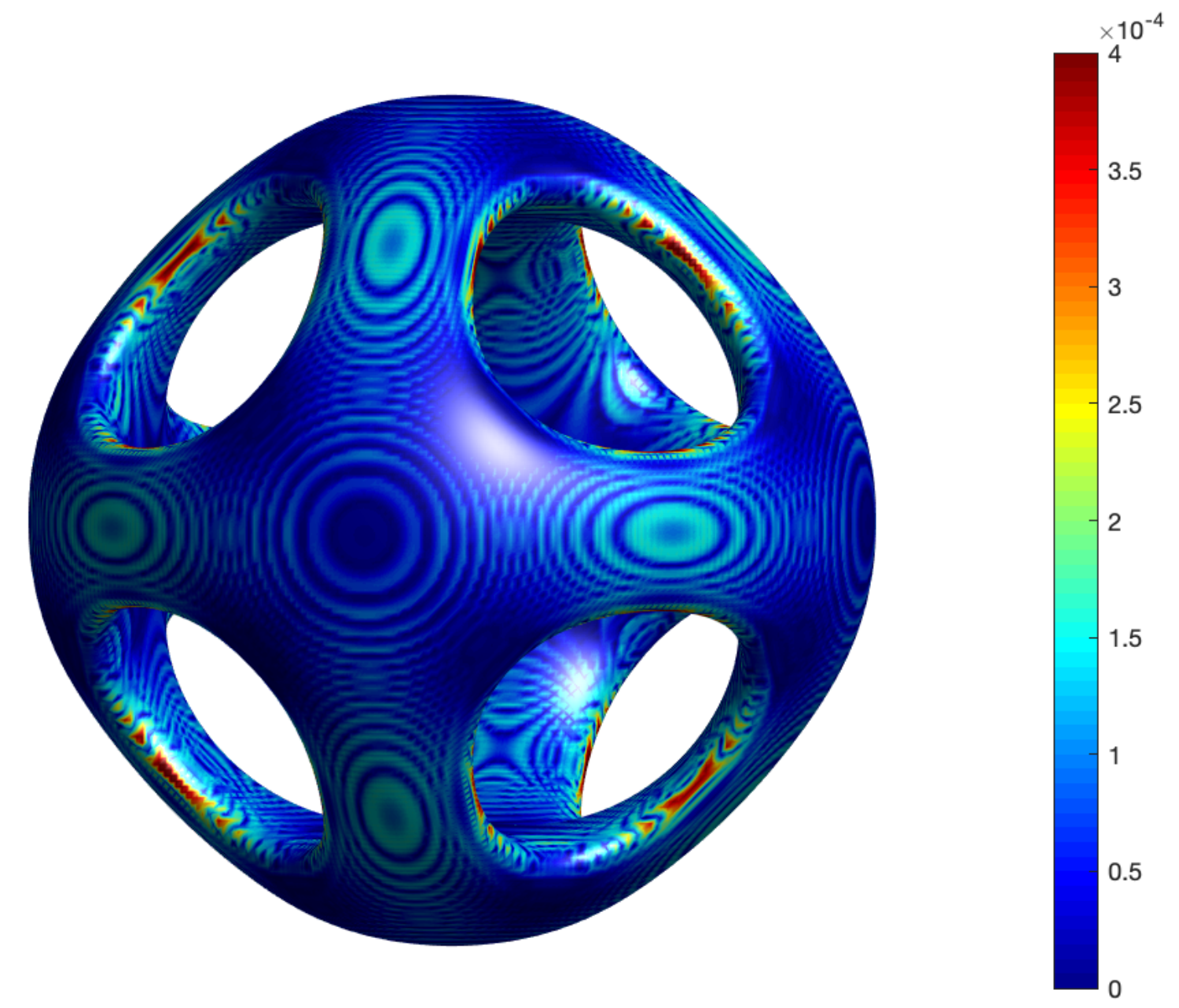}~
\includegraphics[width=.32\textwidth]{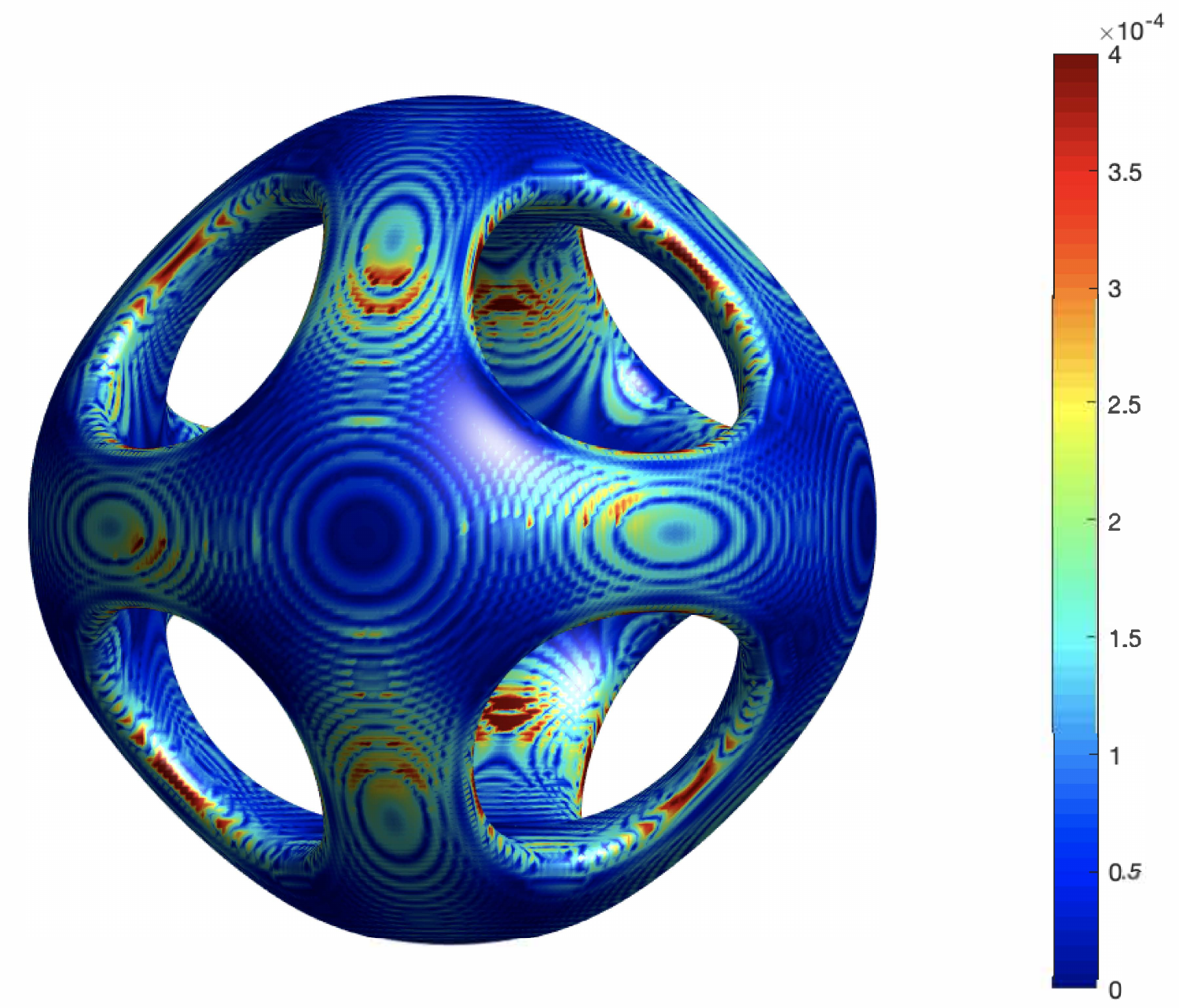}
\caption{A plot of the orthocircle interface (left). The error surfaces of PPIFE solutions with (middle) and without (right) imposing the maximal angle condition (mesh size $N=160$).}
\label{fig: orthocircle}
\end{figure}


\begin{figure}[h!]
\centering
\includegraphics[width=.32\textwidth]{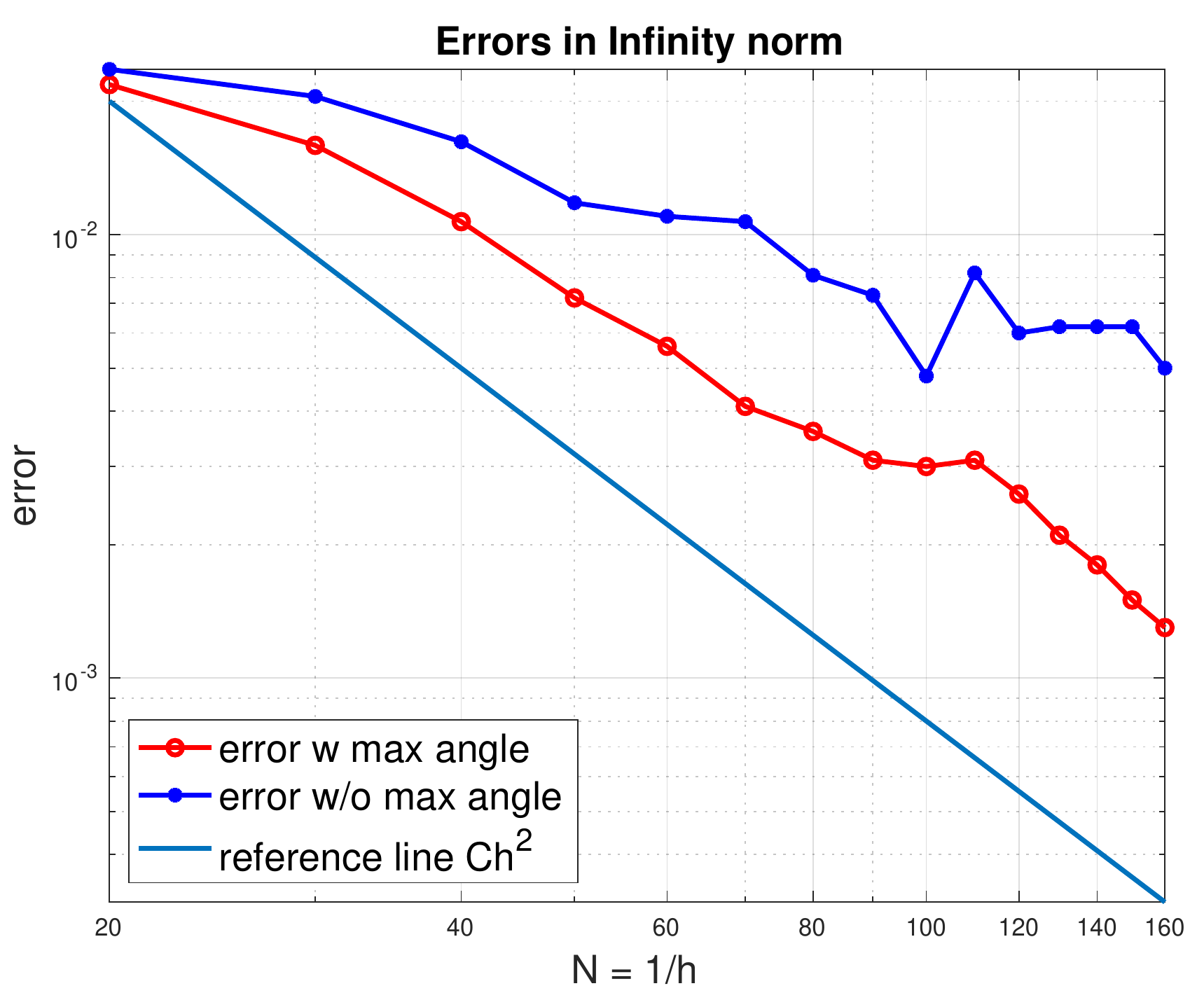}
\includegraphics[width=.32\textwidth]{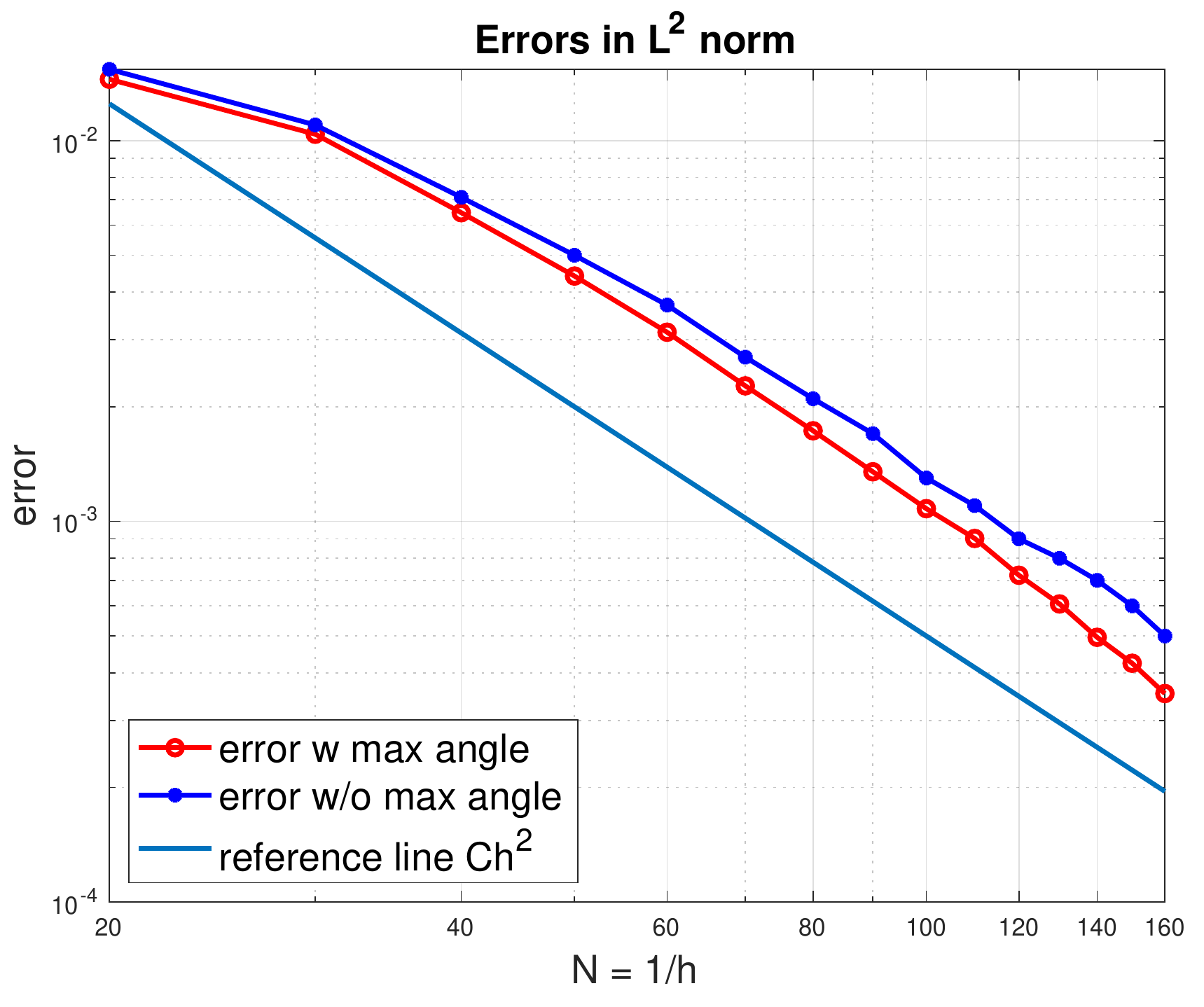}
\includegraphics[width=.32\textwidth]{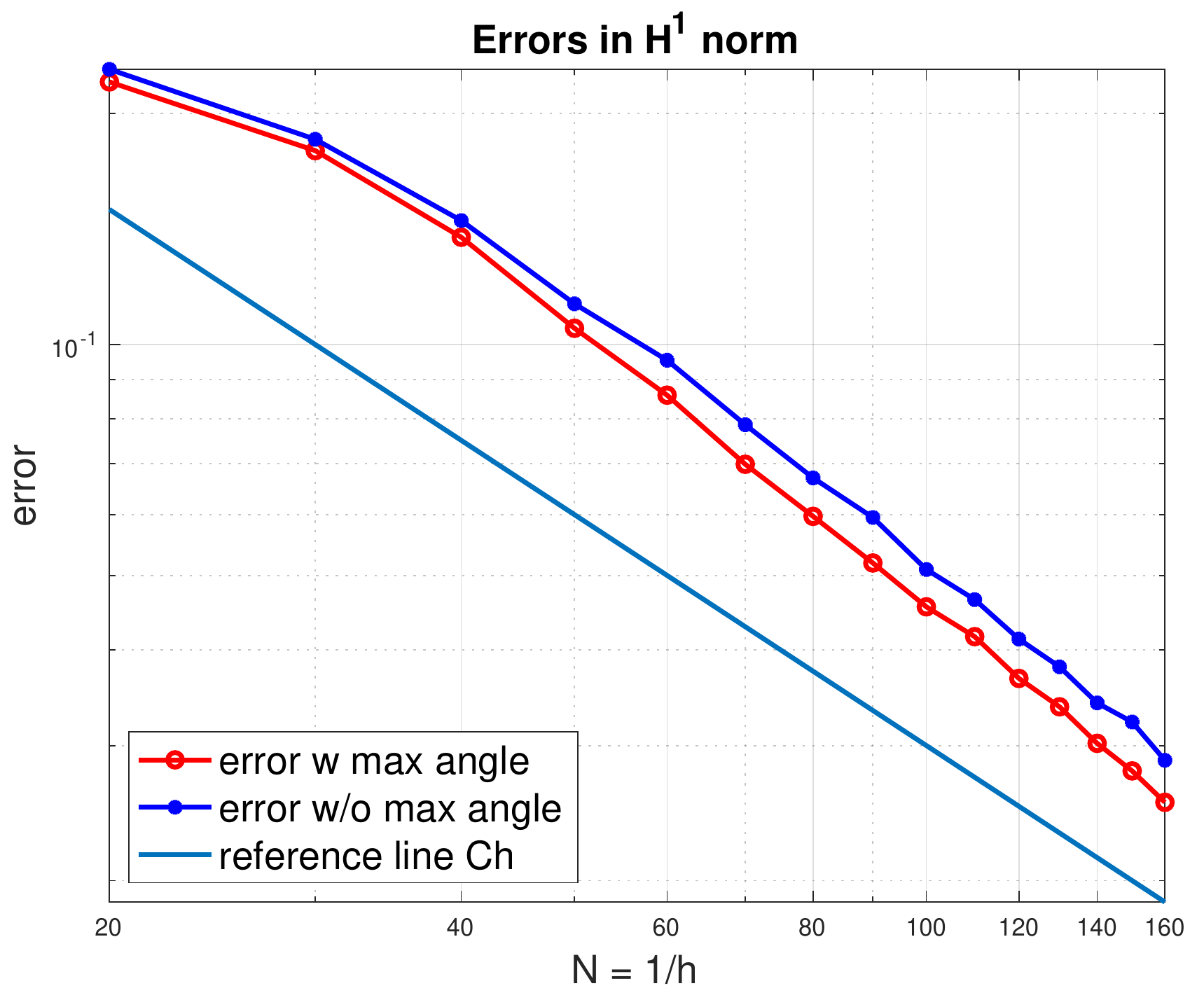}
\caption{Convergence of Example 3}
\label{fig: convergence ex3}
\end{figure}

Moreover, since the only extra work of the IFE method is to replace the standard shape functions by some special shape functions on interface elements, we report the percentage of interface elements over all elements, defined by $|\mathcal{T}_h^i|/|\mathcal{T}_h|$, for all three examples. The number of interface elements is expected to be $\mathcal{O}(N^2)$, and the number of all elements is $\mathcal{O}(N^3)$, so the percentage should be a linear function of the mesh size $h=1/N$. In Figure \ref{fig: interface percentage}, we can observe this linear relationship clearly. Also, as the shape of interface elements becomes more complex from Example 1 to Example 3, the proportionality constant gets larger. However, even for complicated interface shapes, such as the orthocircle in Example 3, there are only less than $3\%$ interface elements on our finest mesh ($N=160$, around 4 million cuboids). As a result, the majority of the computation (over $97\%$ of the total elements) can be done using the standard FEM package.

\begin{figure}[h!]
\centering
\includegraphics[width=.5\textwidth]{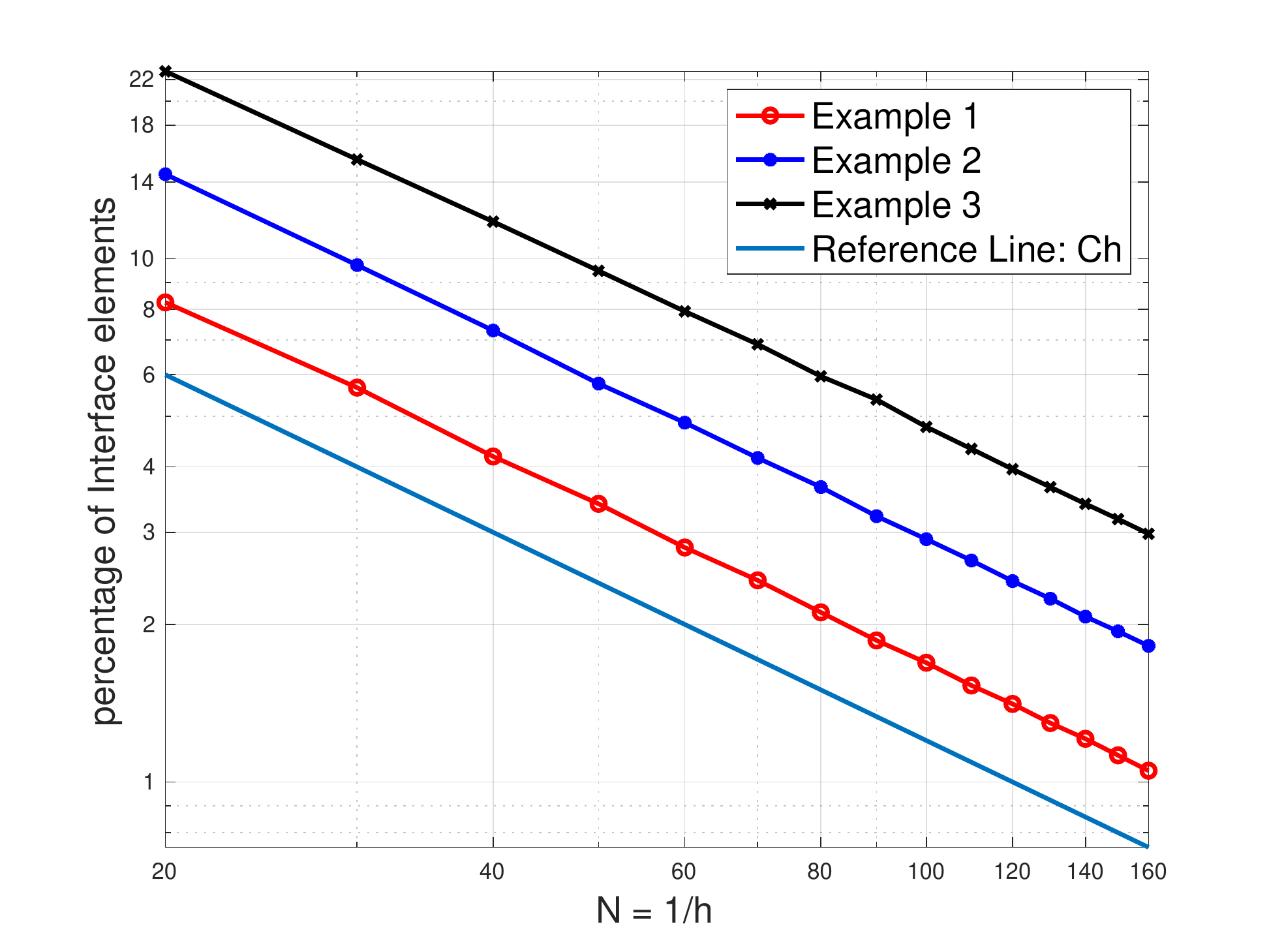}
\caption{Percentage of interface elements for Examples 1-3.}
\label{fig: interface percentage}
\end{figure}

\subsubsection*{Example 4 (A Real-World Interface: Dabbling Duck)}

In this example, we apply our algorithm to a real-world geometric object, the surface of a dabbling duck shown in Figure \ref{fig:duck}. The original data of this interface consists of many cloud points on the surface as shown in the left plot of Figure \ref{fig:duck}. We refer readers to \cite{2011RouhaniSappa} for the availability of the data. We perform the computation on the modeling domain $\Omega=(0.2,1)\times(0.2,1)\times(0.1,0.9)$ which is large enough to contain all the data points. The fundament step in the computation is to generate a smooth surface based on the raw data points. Since only the lowest order accuracy is considered in this article, here we generate a signed-distance function by directly computing the distance from nodes in a given mesh to those data points. Then the zero level-set of the signed-distance function is used as the computational interface in this example.

We consider the equation \eqref{model} with the data $f=0$ in $\Omega$, $u=\sin(3\pi x)\sin(3\pi y)\sin(3\pi z)$ on $\partial \Omega$ and $\beta^-=1$, $\beta^+=10$. The Cartesian mesh with $N^3$ cuboids is generated on $\Omega$ with the mesh size $N=16,32,64,128,256$. We note that it is difficult to construct a function satisfying the homogeneous jump condition exactly on this complicated real-world interface, so here we shall use the numerical solution computed on the finest mesh $N=256$ as the reference solution to compute the errors. The numerical errors and their convergence order are presented in Table \ref{table:duck_err_rate} where we can observe that the numerical solution errors almost have the expected optimal convergence order in $L^2$ and $H^1$ norms. It agrees with the theoretical analysis even for this interface generated from the real-world data.

\begin{figure}[h!]
\centering
\begin{subfigure}{.52\textwidth}
     \hspace{.8in}\includegraphics[width=2.6in]{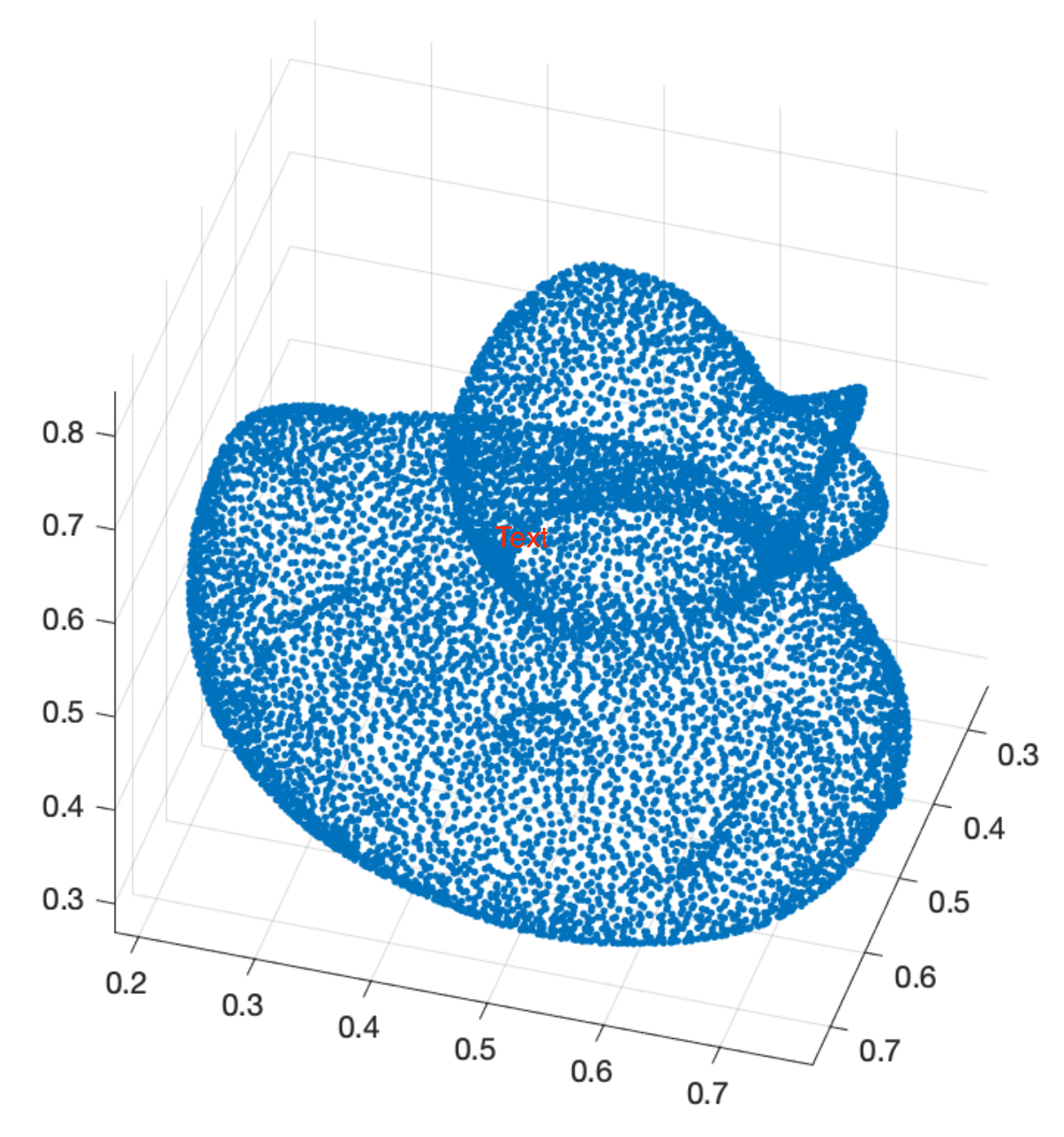}
\end{subfigure}
~
\begin{subfigure}{.45\textwidth}
     \includegraphics[width=2in]{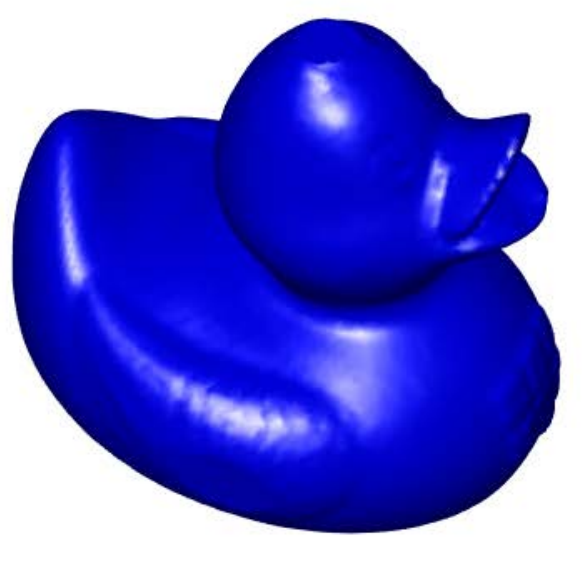}
\end{subfigure}
\caption{A duck-shape interface: cloud points(left) and reconstructed smooth interface (right)}
  \label{fig:duck} 
\end{figure}

\begin{table}[h!]
\centering
\begin{tabular}{c c c c c c c c c}
\hline
      & $e^0_h$ & order & $e^{1}_h$ & order \\\hline 
16   & 5.6418E-3 & NA    & 5.4544E-1 & NA \\ 
32    & 2.0913E-3 & 1.43 & 2.5933E-1 & 1.07 \\ 
64    & 5.1039E-4 & 2.03 & 1.2048E-1 & 1.11 \\ 
128   & 9.5244E-5 & 2.42 & 5.5962E-2 & 1.11 \\\hline
\end{tabular}
\caption{Numerical solution errors and the convergence rates}
\label{table:duck_err_rate}
\end{table}

In addition, we plot the error of IFE solution on $N=128$ at the slices at $y=0.35$ ,$0.4$, $0.45$, $0.5$, $0.55$, and $0.6$ in Figure \ref{fig:duck_err_slice}. We note that the majority of the errors are concentrated on the interface, and the errors are significantly smaller away from the interface. This phenomenon is also observed in Examples 2 and 3 for which the analytical solutions are available. We believe this is due to the advantage of the highly structured mesh such as the Cartesian mesh that the IFE method can use to solve interface problems. In particular, we also note that several spots near the interface have apparently large errors which are the head top, beak, tail and the front and back of the duck base. These portions of the interface certainly have large curvatures, namely the interface is bending severely. To further investigate how the shape of the interface can effect the errors, we plot the relative errors on the interface in Figure \ref{fig:duck_err_interface}. As indicated by these figures, the errors are concentrated on the portion of the interface including the peak, neck (the lower-right plot in Figure \ref{fig:duck_err_interface}), tail (the lower-left plot in Figure \ref{fig:duck_err_interface}) and the surrounding of the base (the upper-right plot in Figure \ref{fig:duck_err_interface})). But we also emphasize that the IFE method performs quite satisfactory on the majority part of the interface surface. The large curvatures can not be avoided in real-world geometric bodies. How to further enhance the performance of IFE methods on the large-bending surface may require local mesh refinement, i.e., some adaptive mesh strategy for IFE method \cite{2019HeZhang}. This could be an interesting topic in our future research.

\begin{figure}[h!]
\centering
\includegraphics[width=.95\textwidth]{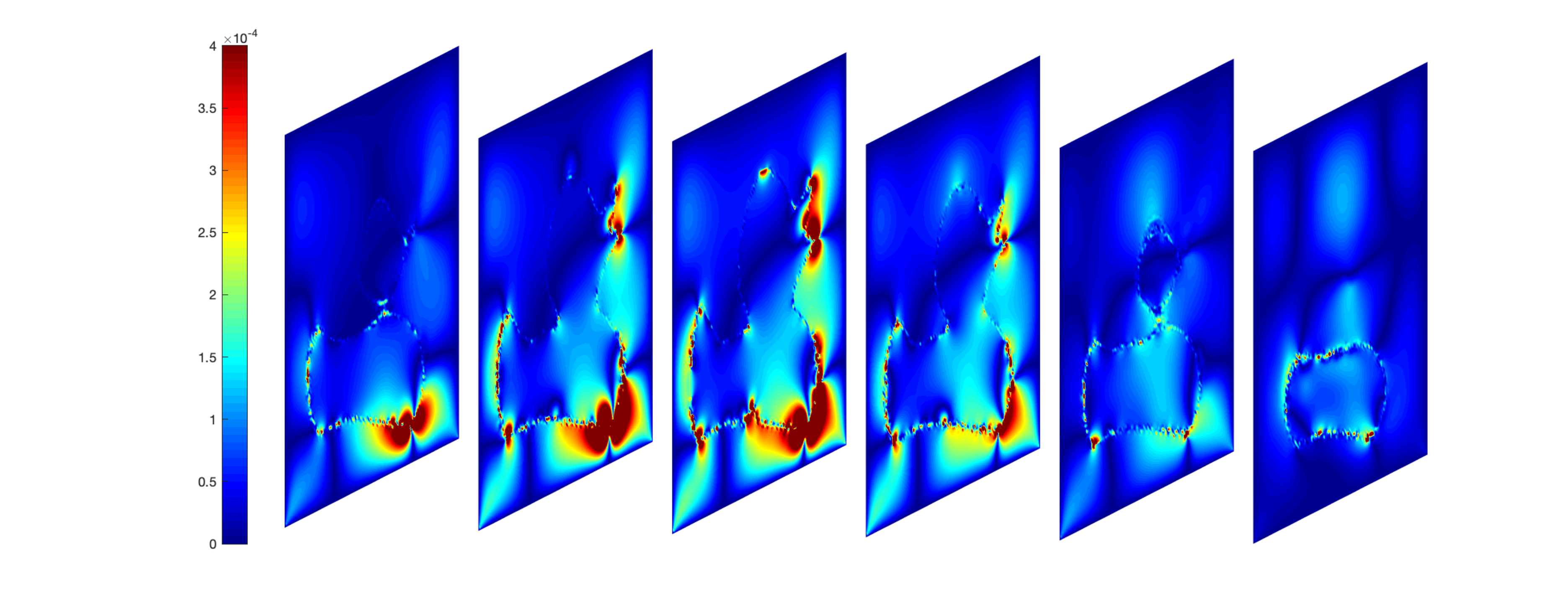}
\caption{The slices of errors between $N=128$ and $N=256$ on $y=0.35,0.4,0.45,0.5,0.55,0.6$}
  \label{fig:duck_err_slice} 
\end{figure}

\begin{figure}[h!]
\centering
\begin{subfigure}{.4\textwidth}
     \includegraphics[width=2in]{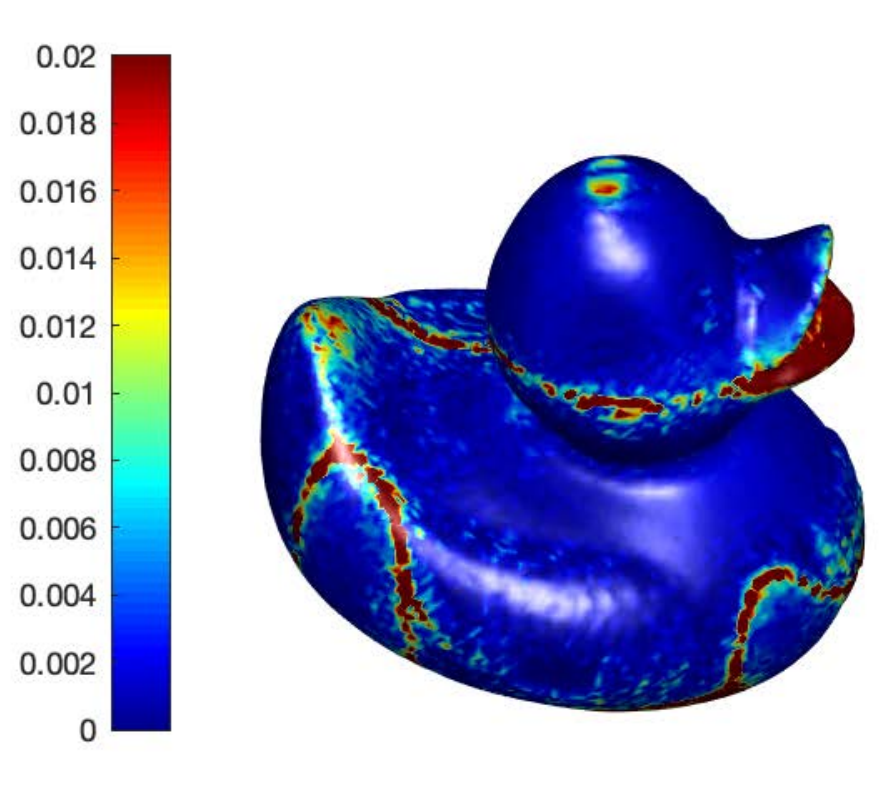}
\end{subfigure}
~~~~~~
\begin{subfigure}{.4\textwidth}
     \includegraphics[width=2in]{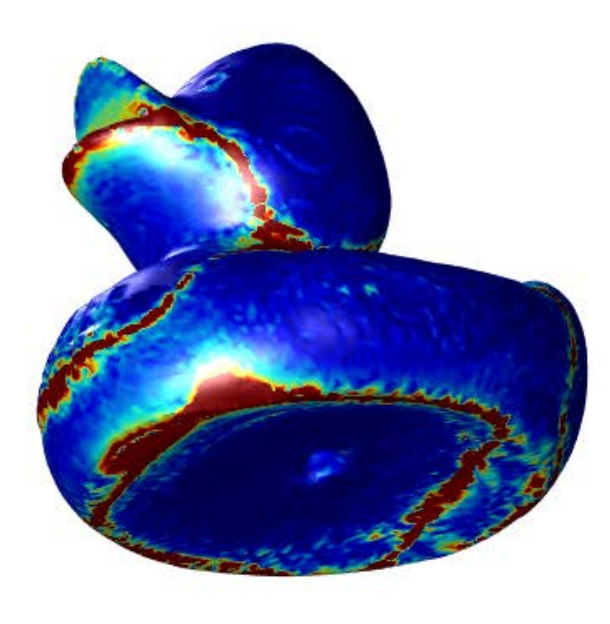}
\end{subfigure}
\vskip\baselineskip
\begin{subfigure}{.4\textwidth}
     \includegraphics[width=2in]{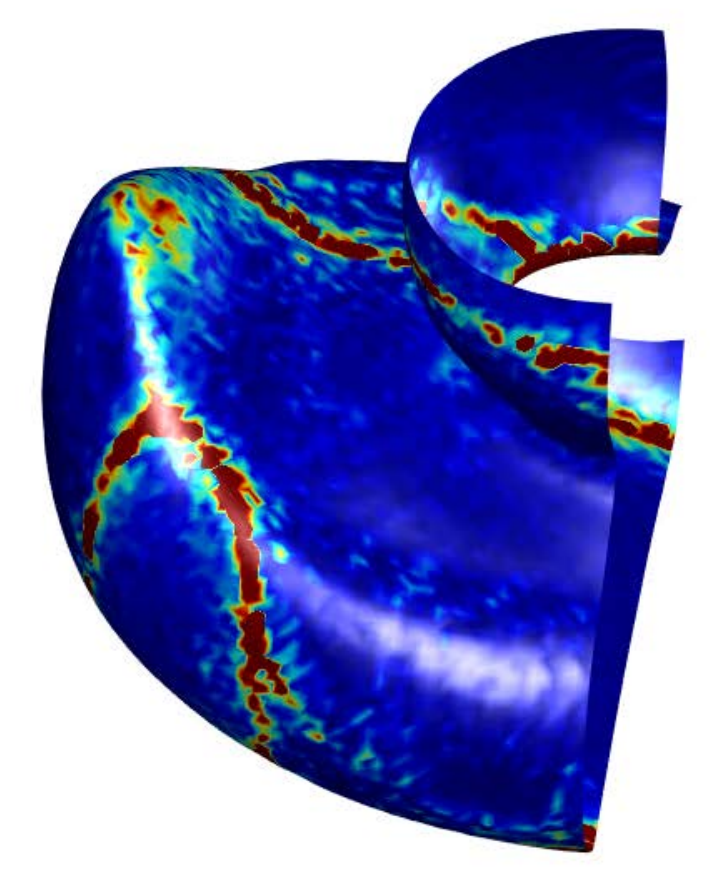}
\end{subfigure}
~~~~~~
\begin{subfigure}{.4\textwidth}
     \includegraphics[width=2in]{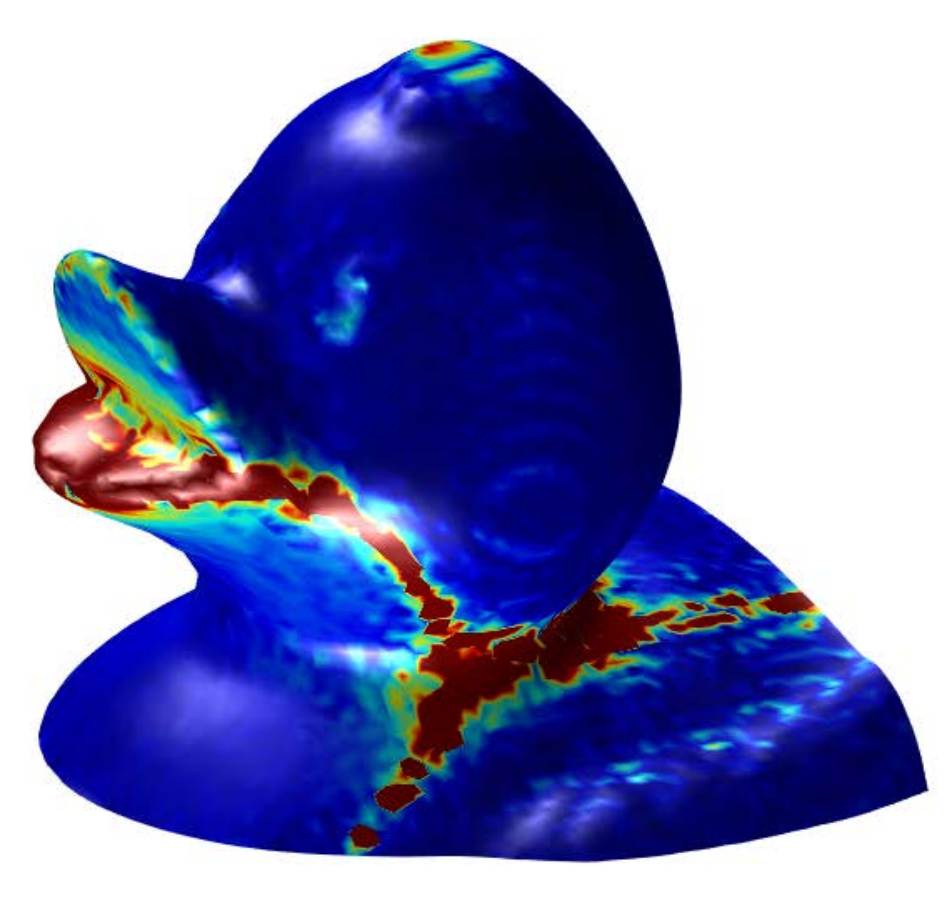}
\end{subfigure}
\caption{Relative solution errors of $N=128$ on interface}
  \label{fig:duck_err_interface} 
\end{figure}

\section{Conclusions}
\label{sec:conclusions}
In this article, we have developed a partially penalized IFE (PPIFE) method for solving elliptic interface problems in three-dimensional space on unfitted meshes. The IFE space is isomorphic to the standard continuous piecewise trilinear finite element space defined on the same mesh, which is independent of the interface location. The penalties are only added on interface faces to handle the discontinuities of IFE functions. We show the PPIFE solutions have optimal convergence rates in both the $L^2$ and $H^1$ norm regardless of interface location. Numerical experiments are performed to validate the theoretical estimates for both artificial interface and real-world interface models.

\section*{Acknowledgments}
The work is partially supported by National Science Foundation Grants DMS-1720425, DMS-2005272.

\end{document}